\documentclass[11pt,oneside]{amsart}
\usepackage{amsmath,amssymb,amsthm}
\usepackage[alphabetic,abbrev]{amsrefs} 
\usepackage{tikz,tikz-cd}

\newcommand{\C}{\mathbb{C}} 
\newcommand{\Ptn}{\mathcal{P}} 
\newcommand{\End}{\operatorname{End}} 
\newcommand{\GL}{\operatorname{GL}} 
\newcommand{\SL}{\operatorname{SL}} 
\newcommand{\ov}{\overline} 
\newcommand{\bE}{\mathbf{E}} 
\newcommand{\bL}{\mathbf{L}} 
\newcommand{\bF}{\mathbf{F}} 
\newcommand{\ee}{\mathbf{e}} 
\newcommand{\ff}{\mathbf{f}} 

\newcommand{\PP}{\mathcal{P}'} 
\newcommand{\tabT}{\mathsf{T}} 
\newcommand{\g}{\mathfrak{g}} 
\newcommand{\gl}{\mathfrak{gl}} 
\newcommand{\Lie}{\operatorname{Lie}} 
\newcommand{\Specht}{\operatorname{Sp}} 
\newcommand{\Sym}{\operatorname{Sym}} 
\newcommand{\id}{\operatorname{id}} 
\newcommand{\J}{\mathcal{J}} 

\newcommand{\sA}{\mathcal{A}} 
\newcommand{\sB}{\mathcal{B}} 
\newcommand{\sU}{\mathcal{U}} 
\newcommand{\dom}{\operatorname{dom}} 
\newcommand{\im}{\operatorname{im}} 
\newcommand{\rank}{\operatorname{rank}} 
\newcommand{\bil}[2]{\langle #1, #2 \rangle} 

\swapnumbers 

\newtheorem{thm}{Theorem}[section]
\newtheorem*{thm*}{Theorem}
\newtheorem{lem}[thm]{Lemma}
\newtheorem*{lem*}{Lemma}
\newtheorem{prop}[thm]{Proposition}
\newtheorem*{prop*}{Proposition}
\newtheorem{cor}[thm]{Corollary}
\newtheorem*{cor*}{Corollary}

\newtheorem*{conj*}{Conjecture}

\theoremstyle{definition}
\newtheorem{defn}[thm]{Definition}
\newtheorem*{defn*}{Definition}
\newtheorem{example}[thm]{Example}
\newtheorem*{example*}{Example}
\newtheorem{rmk}[thm]{Remark}
\newtheorem*{rmk*}{Remark}

\newcommand{\Part}[1]{
 \foreach \x [count=\s from 1] in {#1}{
 	{\ifnum\s=1
		\draw (0,\s-1)--(\x,\s-1); 
		\fi}
   \draw (0,\s) to (\x,\s);
   \foreach \y in {0, ..., \x} {\draw (\y,\s)--(\y,\s-1);}
 }}

\def\UNIT{.18} \newcommand{\PART}[1]{
\begin{tikzpicture}[xscale=\UNIT, yscale=-\UNIT] 
	\Part{#1}
\end{tikzpicture}
}

\renewcommand{\labelenumi}{(\alph{enumi})}
\allowdisplaybreaks

\title[Schur--Weyl duality for the Burau representation]%
{Schur--Weyl duality for tensor powers of the Burau representation}
\author{Stephen Doty}
\email{doty@math.luc.edu, tonyg@math.luc.edu}
\author{Anthony Giaquinto}
\address{Department of Mathematics and Statistics,
  Loyola University Chicago, Chicago, IL 60660 USA}

\begin{document}
\begin{abstract}\noindent
Artin's braid group $B_n$ is generated by $\sigma_1, \dots,
\sigma_{n-1}$ subject to the relations
\[
\sigma_i \sigma_{i+1} \sigma_i = \sigma_{i+1} \sigma_i \sigma_{i+1},
\quad \sigma_i\sigma_j = \sigma_j \sigma_i \text{ if } |i-j|>1.
\]
For complex parameters $q_1,q_2$ such that $q_1q_2 \ne 0$, the group
$B_n$ acts on the vector space $\bE = \sum_i \C \ee_i$ with basis
$\ee_1, \dots, \ee_n$ by
\begin{gather*}
\sigma_i \cdot \ee_i = (q_1+q_2)\ee_i + q_1\ee_{i+1}, \quad \sigma_i
\cdot \ee_{i+1} = -q_2\ee_i, \\ \sigma_i \cdot \ee_j = q_1 \ee_j
\text{ if } j \ne i,i+1.
\end{gather*}
This representation is (a slight generalization of) the Burau
representation. If $q = -q_2/q_1$ is not a root of unity, we show that
the algebra of all endomorphisms of $\bE^{\otimes r}$ commuting with
the $B_n$-action is generated by the place-permutation action of the
symmetric group $S_r$ and the operator $p_1$, given by
\[
p_1(\ee_{j_1} \otimes \ee_{j_2} \otimes \cdots \otimes \ee_{j_r}) =
q^{j_1-1} \, \sum_{i=1}^n \ee_i \otimes \ee_{j_2} \otimes
\cdots \otimes \ee_{j_r} .
\]
Equivalently, as a $(\C B_n, \PP_r([n]_q))$-bimodule, $\bE^{\otimes r}$
satisfies Schur--Weyl duality, where $\PP_r([n]_q)$ is a certain
subalgebra of the partition algebra $\Ptn_r([n]_q)$ on $2r$ nodes with
parameter $[n]_q = 1+q+\cdots + q^{n-1}$, isomorphic to the semigroup
algebra of the ``rook monoid'' studied by W.~D.~Munn, L.~Solomon, and
others.
\end{abstract}
\maketitle

\section{Introduction}\noindent
Although most results hold over any field of characteristic zero, we
will work over the complex field $\C$ throughout this paper, so we
abbreviate $\otimes_\C$ to $\otimes$ and $\dim_\C$ to $\dim$, etc. The
partition algebra $\Ptn_r(n)$ (see Section \ref{sec:Ptn-alg}) was
introduced independently by P.~P.~Martin \cites{Martin:book,
  Martin:94} and V.~F.~R.~Jones \cite{Jones:94} in connection with the
Potts model in mathematical physics, as the generic centralizer
algebra $\End_{W_n}(E^{\otimes r})$ for the natural permutation
representation $E$ of the Weyl group $W_n$ of $\GL(E) \cong
\GL_n(\C)$, acting diagonally on $E^{\otimes r}$. By \cite{Jones:94}
there is a Schur--Weyl duality for $E^{\otimes r}$ as a $(\C W_n,
\Ptn_r(n))$-bimodule; this was recently applied \cite{BDO:15} to
explain stability properties of Kronecker coefficients and Deligne
\cite{Deligne} gave a categorical framework for these algebras.

In this paper, we replace $E$ by the (unreduced) Burau representation
$\bE$ of Artin's braid group $B_n$ and ask for a combinatorial
description of the centralizer algebra $\End_{B_n}(\bE^{\otimes r})$,
with $B_n$ acting diagonally on the tensor power. The paper
\cite{Jones:87} focused attention on the class of representations of
$B_n$ in which the image $T_i$ of the braid group generators
$\sigma_i$ satisfy a quadratic relation $(T_i-q_1)(T_i-q_2)=0$. The
Burau representation $\bE$ is the simplest non-trivial example of
such; we choose to work with its general two-parameter version defined
in Section \ref{sec:braid}. Most of our results depend not on
$q_1,q_2$ but only on their negative ratio $q = -q_2/q_1$.

If $q$ is not a root of unity, our main result, Theorem
\ref{thm:main}, is that $\bE^{\otimes r}$ satisfies Schur--Weyl
duality as a $(\C B_n, \PP_r([n]_q))$-bimodule, where $\PP_r([n]_q)$
is the subalgebra of the partition algebra $\Ptn_r([n]_q)$ spanned by
partial permutation diagrams, at parameter the quantum integer $[n]_q
= 1+q+\cdots+q^{n-1}$. At $q=1$ the Burau representation is isomorphic
to the natural permutation representation of the Weyl group $W_n$
($\cong S_n$) of $\GL(\bE)$, and generically $\End_{W_n}(\bE^{\otimes
  r}) \cong \Ptn_r(n)$, so we think of $\PP_r([n]_q)$ as a
$q$-analogue of the partition algebra $\Ptn_r(n)$; however, it is not
a flat deformation since the two algebras have different dimensions. A
different $q$-analogue was studied in \cite{Halverson-Thiem}; it does
have the same dimension as the partition algebra.

Corollary \ref{cor:SWD:q=1} gives a second new instance of Schur--Weyl
duality involving the braid group, in which the algebra $\PP_r([n]_q)$
is replaced by its specialization $\PP_r(n)$ at $q=1$. 

The key technical result (Theorem \ref{thm:density}) is topological,
which states that as long as $q = -q_2/q_1$ is not a root of unity,
then the Zariski-closure of the image of the \emph{reduced} Burau
representation $B_n \to \GL(\bF)$ contains $\SL(\bF)$. This leads to
our third new instance of Schur--Weyl duality (Theorem
\ref{thm:SWD-R^k}) for the space $\bF^{\otimes k}$ regarded as a
representation of the group $B_n \times S_k$, with $B_n$ acting
diagonally and $S_k$ acting by place-permutation.

In Section \ref{sec:P'}, we study the algebra $\PP_r(z)$ for $z \ne 0$
an arbitrary complex parameter, from the viewpoint of iterated
inflations in the theory of cellular algebras, which enables a quick
proof of its semisimplicity and an easy derivation of its irreducible
representations.  Theorem \ref{thm:present}, the main result of
Section \ref{sec:P'}, gives a new presentation of $\PP_r(z)$ by
generators and relations. As a consequence (Corollary
\ref{cor:present}) we deduce that $\PP_r(z) \cong \PP_r(1)$ for any $z
\ne 0$.  In particular, $\PP_r([n]_q)$ is isomorphic to the semigroup
algebra of the rook monoid (symmetric inverse semigroup) studied by
Munn \cites{Munn:57a,Munn:57b}, Solomon \cite{Solomon}, and others.

Proposition \ref{prop:splitting} implies that $\bE$ is a semisimple
$\C B_n$-module if and only if $[n]_q \ne 0$. So the assumption $[n]_q
\ne 0$ implies that $\bE^{\otimes r}$ is also semisimple as a $\C
B_n$-module, by \cite{Chevalley}*{p.~88}; thus by the Jacobson density
theorem \cite{Jacobson}*{\S4.3} or \cite{Lang}*{Chap.~XVII, Theorem
  3.2} it follows that $\bE^{\otimes r}$ satisfies the
double-centralizer property (the natural map from $\C B_n
\to \End_{Z}(\bE^{\otimes r})$ is surjective, where $Z
= \End_{B_n}(\bE^{\otimes r})$). Thus the primary task is to identify
the centralizer $Z = \End_{B_n}(\bE^{\otimes r})$ combinatorially,
which we do when $q$ is not a root of unity. If $q$ is an $l$th root
of unity for $l \ne n$ but $[n]_q \ne 0$ then a description of $Z$ is
an interesting open problem.

\subsection*{Acknowledgments}
The authors are grateful to Greg Kuperberg for useful discussions
regarding the proof of Theorem \ref{thm:density}, given in Appendix
\ref{sec:proof}, and to Stephen Donkin and Darij Grinberg for helpful
suggestions.

\section{Artin's braid group}\label{sec:braid}\noindent
Artin's braid group $B_n$ is the group defined by generators
$\sigma_1, \dots, \sigma_{n-1}$ subject to the braid relations
\begin{equation}\label{eq:braid-rels}  
\sigma_i \sigma_{i+1}\sigma_i = \sigma_{i+1}\sigma_i \sigma_{i+1}, \quad
\sigma_i\sigma_j = \sigma_j\sigma_i \text{ if } |i-j| > 1.
\end{equation}
The classical $n$-dimensional (unreduced) Burau representation $B_n
\to \GL(\bE)$, introduced in \cite{Burau}, has a number of definitions
\cite{BLM}. An algebraic definition (see \cite{Jones:87}) may be given
in terms of a basis $\ee_1, \dots, \ee_n$ by letting $\sigma_i$ act by
\begin{equation}\label{eq:Burau}
  \sigma_i \cdot \ee_j =
  \begin{cases}
    \ee_j & \text{ if } j \ne i, i+1 \\
    t \ee_{j-1} & \text{ if } j = i+1 \\
    (1-t) \ee_j + \ee_{j+1} &\text{ if } j = i
  \end{cases}
\end{equation}
where $0 \ne t \in \C$.
In other words, if $Q' = \begin{bmatrix} 1-t & t \\ 1 & 0
  \end{bmatrix}$ then $\sigma_i$ acts via the $n \times n$
block matrix
\[ \sigma_i \mapsto 
\begin{bmatrix}
  I_{i-1} & 0 & 0 \\
  0 & Q' & 0 \\
  0 & 0 & I_{n-i-1}
\end{bmatrix}
\]
where as usual $I_k$ denotes the $k \times k$ identity matrix. 

We now (slightly) generalize the Burau representation to depend on two
parameters $q_1,q_2$ by setting $t = -q_2/q_1$ in equations
\eqref{eq:Burau} and scaling by $q_1$ to obtain
\begin{equation}\label{eq:Burau2}
  \sigma_i \cdot \ee_j =
  \begin{cases}
    q_1 \ee_j & \text{ if } j \ne i, i+1 \\
    -q_2 \ee_{j-1} & \text{ if } j = i+1 \\
    (q_1+q_2) \ee_j + q_1 \ee_{j+1} &\text{ if } j = i .
  \end{cases}
\end{equation}
In other words, if we define
$Q = \begin{bmatrix} q_1+q_2 & -q_2
  \\ q_1 & 0
\end{bmatrix}$
then $\sigma_i$ acts on $\bE$ via the $n \times n$
block diagonal matrix
\[ \beta_i = 
\begin{bmatrix}
  q_1 I_{i-1} & 0 & 0 \\
  0 & Q & 0 \\
  0 & 0 & q_1 I_{n-i-1}
\end{bmatrix} .
\]
The map defined by $\sigma_i \mapsto \beta_i$ is a representation $B_n
\to \GL(\bE)$ of the braid group $B_n$. Its linear extension $\C B_n
\to \End(\bE)$ to the group algebra $\C B_n$ factors through (via
$\sigma_i \mapsto T_i \mapsto \beta_i$) the two-parameter
Iwahori--Hecke algebra $H_n(q_1,q_2)$, defined (as in \cite{Bigelow})
by generators $T_1, \dots, T_{n-1}$ subject to the relations
\begin{gather}
  T_iT_{i+1}T_i = T_{i+1}T_iT_{i+1}, \quad T_iT_j=T_jT_i \text{ if } |i-j|>1 \\
  (T_i-q_1)(T_i-q_2)=0.
\end{gather}
We always assume that $q_1q_2 \ne 0$, so that the
generators $T_i$ are invertible elements in $H_n(q_1,q_2)$, with
\[
{T_i}^{-1} = (T_i-q_1-q_2)/(q_1q_2). 
\]
Although it is easy to eliminate one of the parameters, we prefer to
carry both along, as explained in the remark below. Doing so causes no
essential difficulty.

\begin{rmk}\label{rmk:specializations}
Set $q = -q_2/q_1$.  There are well-known algebra isomorphisms \[
H_n(q_1,q_2) \cong H_n(-1,q),\quad H_n(q_1,q_2) \cong H_n(1,-q)\]
defined by sending $T_i \mapsto -q_1T_i$, $T_i \mapsto q_1T_i$
respectively.  Moreover, the map $T_i \mapsto q^{-1/2}T_i$ defines an
algebra isomorphism \[H_n(-q^{-1/2}, q^{1/2}) \cong H_n(-1,q),\] so
$H_n(-q^{-1/2}, q^{1/2}) \cong H_n(q_1,q_2)$. The algebra
$H_n(-q^{-1/2}, q^{1/2})$, the ``balanced'' form of the Iwahori--Hecke
algebra, is often preferred in the theory of quantum groups.  The
generalized Burau representation makes equal sense in all three of
these popular one-parameter versions of $H_n(q_1,q_2)$.
\end{rmk}

\section{Decomposing the Burau representation}\label{sec:Hecke}\noindent
In this section we introduce a bilinear form in order to decompose
$\bE$ as a $\C B_n$-module under appropriate conditions. The algebra
$H_n(q_1,q_2)$ is used as a technical aid in obtaining the
decomposition; it will not be used anywhere else in this paper.

By direct computation, we notice the following explicit eigenvectors
for the $\beta_i$ operators defined in the previous section.

\begin{lem}\label{lem:eigen}
  Assume that $q_1q_2 \ne 0$.  For any $i = 1, \dots, n-1$ the operator
  $\beta_i$ has eigenvectors:
  \begin{enumerate}
    \item $\ee_1, \dots, \ee_{i-1}, \ee_i+\ee_{i+1}, \ee_{i+2}, \dots,
      \ee_n$ with eigenvalue $q_1$.
    
  \item $\ff_i := q_2\ee_i + q_1\ee_{i+1}$ with eigenvalue $q_2$.
  \end{enumerate}
  In particular, $\beta_i$ is diagonalizable if and only if $q_1 \ne
  q_2$ and $\ff_0 := \ee_1+\cdots+\ee_n$ is a simultaneous eigenvector
  for all the $\beta_i$.
\end{lem}

\begin{proof}
Parts (a), (b) are easily checked. Observe that $\ff_i = q_2\ee_i +
q_1\ee_{i+1}$ and $\ee_i+\ee_{i+1}$ are linearly dependent if and only
if $q_1=q_2$, which proves the diagonalizability claim.
\end{proof}

\begin{rmk}
If $q_1=q_2$ then the $\beta_i$ have only one eigenvalue and the
corresponding eigenspace has dimension $n-1$. 
\end{rmk}

By Lemma \ref{lem:eigen}, the simultaneous eigenvector $\ff_0$ spans a
line
\[
\bL = \C \ff_0 = \C(\ee_1 + \cdots + \ee_n) \subset \bE.
\]
The line $\bL$ is an $H_n(q_1,q_2)$-submodule, hence also a $\C
B_n$-submodule, of $\bE$.  Let $\bF$ be the subspace
\[
\bF = \sum_{i=1}^{n-1} \C \ff_i = \sum_{i=1}^{n-1} \C (q_2\ee_i
+ q_1\ee_{i+1})
\]
spanned by the eigenvectors in Lemma \ref{lem:eigen}(b).  Since by
assumption $q_1$, $q_2$ are nonzero, the spanning vectors are linearly
independent, so $\bF$ is a subspace of $\bE$ of dimension $n-1$. We
leave the elementary proof of the following to the reader.

\begin{lem}\label{lem:action-on-R}
  For any $i,j = 1, \dots, n-1$ the action of $\sigma_i \in B_n$ on
  $\ff_j \in \bF$ is given by the rules
  \begin{enumerate}
  \item  $\sigma_i \cdot \ff_j = q_1 \ff_j$  if $j \ne i-1, i, i+1$.
  \item  $\sigma_i \cdot \ff_{i-1} = q_1 \ff_{i-1} + q_1\ff_i$.
  \item  $\sigma_i \cdot \ff_i = q_2\ff_i$.
  \item  $\sigma_i \cdot \ff_{i+1} = -q_2 \ff_i + q_1\ff_{i+1}$.
  \end{enumerate}
  Hence $\bF$ is a $\C B_n$-submodule of $\bE$. Since the action of
  $\C B_n$ on $\bE$ factors through $H_n(q_1,q_2)$, $\bF$ is also an
  $H_n(q_1,q_2)$-submodule of $\bE$, with the $T_i$ acting on $\bF$ by
  the same formulas.
\end{lem}

We call $\bF$ the (generalized) \emph{reduced} Burau representation.
For concreteness, the matrices of the action of the $\sigma_i$ with
respect to the $\{\ff_j\}$-basis of $\bF$ are listed below:
\begin{align*}
\sigma_1 &\mapsto 
\begin{bmatrix}
  q_2 & -q_2 &  0  & \cdots & 0\\
  0   &  q_1 &  0  & \cdots & 0\\
  0   &  0   & q_1 & \cdots & 0\\
  \vdots&\vdots&\vdots&\ddots&\vdots\\
  0   &  0   & 0 & \cdots & q_1
\end{bmatrix} , \qquad
\sigma_{n-1} \mapsto
\begin{bmatrix}
  q_1 & \cdots & 0 & 0 & 0 \\
  \vdots & \ddots & \vdots & \vdots & \vdots \\
  0 & \cdots & q_1 & 0 & 0 \\
  0 & \cdots & 0 & q_1 & 0 \\
  0 & \cdots & 0 & q_1 & q_2 \\
\end{bmatrix}, \\
\sigma_i & \mapsto
\begin{bmatrix}
  q_1 & \cdots & 0 & 0  &  0  & \cdots & 0\\
  \vdots & \ddots & \vdots & \vdots & \vdots & \ddots & 0\\
  0 & \cdots & q_1 & 0  &  0  & \cdots & 0\\
  0 & \cdots & q_1 & q_2 &  -q_2  & \cdots & 0\\
  0 & \cdots & 0   &  0   & q_1 &  \cdots & 0\\
  \vdots & \ddots &\vdots&\vdots&\vdots&\ddots&\vdots\\
  0 & \cdots & 0   &  0   & 0 & \cdots & q_1
\end{bmatrix} \qquad (1 < i < n-1) 
\end{align*}
where the unique diagonal $q_2$-entry is in the $i$th row and column
of the matrix of $\sigma_i$.

\begin{rmk}\label{rmk:det}
Note that the determinant of the $\sigma_i$ on $\bF$ is $q_1^{n-2}q_2$.
\end{rmk}

Next we aim to show that $\bE=\bL \oplus \bF$ under suitable
hypotheses. To that end, we define a symmetric bilinear form
$\bil{-}{-}$ on $\bE$ by
\begin{equation}
  \bil{\ee_i}{\ee_j} = \delta_{ij} q^{j-1}
\end{equation}
extended bilinearly, where $q = -q_2/q_1$.  Let $J = \text{diag}(1, q,
\dots, q^{n-1})$ be the matrix of the form. Since $q_1q_2 \ne 0$ it
follows that $\det J \ne 0$, so the form is nondegenerate.

\begin{lem}\label{lem:perp}
Suppose that $q_1q_2 \ne 0$. Then $\bF = \bL^\perp$, the orthogonal
complement with respect to the form.
\end{lem}

\begin{proof}
We have $\bil{\ff_0}{\ff_i} = \bil{q_2\ee_i+q_1\ee_{i+1}}{\ff_0} =
q_2q^{i-1} + q_1q^{i} = 0$ for all $i = 1, \dots, n-1$. Hence $\bF
\subset \bL^\perp$. Since $\dim \bL^\perp = n-1$ by the standard
theory of bilinear forms, it follows by dimension comparison that the
inclusion is equality.
\end{proof}

\begin{prop}\label{prop:splitting}
  Suppose that $q_1q_2 \ne 0$. Set $q = -q_2/q_1$ and
  $[n]_q = 1 + q + \cdots + q^{n-1}$.
  \begin{enumerate}
  \item $\bE = \bL \oplus \bF$ if and only if $[n]_q \ne 0$. When this
    holds, it is an orthogonal decomposition as
    $H_n(q_1,q_2)$-modules, hence also as $\C B_n$-modules.
  \item If $n > 2$ then $\bF$ is irreducible as a $\C B_n$-module if
    and only if $[n]_q \ne 0$. (If $n=2$ then $\bF$ is irreducible for
    any $q$.)
  \end{enumerate}
\end{prop}

\begin{proof}
(a) We have $\bE = \bL \oplus \bF$ if and only if $\bL \cap \bF = 0$.
As $\dim \bL = 1$ and $\bF=\bL^\perp$, this fails if and only if $\bL
\subset \bF$, i.e., $\bL \subset \bL^\perp$, which is true if and only
if the restriction of the form to $\bL$ is degenerate (i.e., $\bL$ is
a degenerate subspace). Since
\[
\bil{\ff_0}{\ff_0} = \bil{\ee_1+ \cdots + \ee_n}{\ee_1+ \cdots +
  \ee_n} = 1+q+\cdots + q^{n-1} = [n]_q
\]
it follows that the subspace $\bL$ is degenerate if and only if
$[n]_q=0$.

(b) First, observe that it is immediate from parts (b), (d) of Lemma
\ref{lem:action-on-R} that any $\C B_n$-submodule of $\bF$ containing
any one of the basis elements $\ff_i$ must be equal to $\bF$ itself.

Now let $S$ be a proper $\C B_n$-submodule of $\bF$. It is not the
zero module, so it contains at least one nenzero vector
\[
  \mathbf{x} = x_1 \ff_1 + \cdots + x_{n-1} \ff_{n-1} 
\]
where the $x_i \in \C$ are not all zero. Applying the formulas in
Lemma \ref{lem:action-on-R} we see for any $i = 1, \dots, n-1$ that
\begin{align*}
-q_1\mathbf{x} + \sigma_i \cdot \mathbf{x} &= \big( (q_2-q_1)x_i + q_1
x_{i-1} - q_2 x_{i+1} \big)\, \ff_i \\
 &= q_1\big( (-q-1)x_i + x_{i-1} + qx_{i+1} \big) \, \ff_i
\end{align*}
where we stipulate that $x_0 = x_n = 0$ to make the edge cases $i=1,
n-1$ sensible. If the factor $(-q-1)x_i + x_{i-1} - q x_{i+1} \ne 0$
for any $i$ then $\ff_i \in S$ and thus $S$ must contain $\bF$ (by the
previous paragraph), so $S = \bF$. In other words, in order for a
proper submodule $S$ to exist, it is necessary that
\[
(-q-1)x_i + x_{i-1} + q x_{i+1} = 0 \text{ for all } i = 1, \dots, n-1
\]
so the $x_i$ satisfy a homogeneous linear system, and hence (as the
$x_i$ are not all zero) the determinant of its coefficient matrix must
be zero. The coefficient matrix in question is a banded tridiagonal
$(n-1) \times (n-1)$ matrix of the form
\[
A_n = 
\begin{bmatrix}
  a & b \\
  c & a & b \\
  & \ddots & \ddots & \ddots\\
  &  & c & a & b \\
  & & & c & a
\end{bmatrix}
\]
where $a= -(q+1)$, $b=q$, $c= 1$. Determinants of tridiagonal matrices
are known as \emph{continuants} and satisfy the simple recursion
\cite{Muir}*{Chapter~XIII}
\begin{equation}\label{eq:tri-rec}
\det A_n = a \det A_{n-1} - bc \det A_{n-2}.
\end{equation}
(This can be checked directly by basic linear algebra.)  In the
present case this amounts to the formula $D_n = -(q+1)D_{n-1} -
qD_{n-2}$ where we have set $D_n = \det A_n$.  Since $D_3 = [3]_q$, $D_4
= -[4]_q$ it follows by induction that $D_n = (-1)^{n+1}[n]_q$ for all
$n \ge 3$. We conclude that in order to have a proper submodule of
$\bF$ it is necessary that $[n]_q = 0$. That is, $[n]_q \ne 0$ implies
that $\bF$ is irreducible.

Conversely, if $[n]_q = 0$ then by the proof of part (a) we have $\bL
\subset \bF$, and $\bL$ is a proper submodule since $\dim \bF > 1$.
\end{proof}

\begin{rmk}
(i) Notice that the decomposition $\bE = \bL \oplus \bF$ holds at
  $q=1$.

(ii) The proof shows that $\bL \subset \bF$ whenever $[n]_q=0$. 
%
%
\end{rmk}

%
%
%
%

\section{The centralizer $\End_{B_n}(\bE)$}\noindent
Henceforth we set $[n]_q = 1+q+\cdots+q^{n-1}$ and assume that $[n]_q
\ne 0$, where $q = -q_2/q_1$, so that $\bE = \bL \oplus \bF$. Note
that $[n]_q=0$ implies that $q^n=1$ and thus that $q$ is a root of
unity. So if $q$ is not a root of unity then $[n]_q \ne 0$. We wish to
compute the algebra $\End_{B_n}(\bE)$ of operators commuting with the
image of the generalized Burau representation $B_n
\to \End(\bE)$. This is the algebra of all $X \in \End(\bE)$
satisfying
\[
\beta_i X = X \beta_i \iff \beta_i X \beta_i^{-1} = X
\]
for all $i = 1, \dots, n-1$.

Let $P_0$ in $\End(\bE)$ be the orthogonal projection onto $\bL$,
defined by sending $v \mapsto v_0$, where $v = v_0+v'_0$ (uniquely)
for $v_0 \in \bL$, $v'_0 \in \bF$.

\begin{thm}\label{thm:P}
Assume that $q_1q_2 \ne 0$ and $[n]_q \ne 0$, where $q =
-q_2/q_1$. The algebra $\End_{B_n}(\bE)$ of endomorphisms commuting
with the generalized Burau action of $B_n$ on $\bE$ is spanned by the
identity operator $1 = \id_\bE$ and the operator defined by the $n
\times n$ matrix
\[
P = 
\begin{bmatrix}
  1 & q & q^2 & \cdots & q^{n-1} \\
  1 & q & q^2 & \cdots & q^{n-1} \\
  \vdots & \vdots & \vdots & \ddots & \vdots \\
  1 & q & q^2 & \cdots & q^{n-1} 
\end{bmatrix} .
\] 
\end{thm}

\begin{proof}
Since $\bE = \bL \oplus \bF$ as $\C B_n$-modules, it follows from
Schur's Lemma that
\[
\End_{B_n}(\bE) \cong \End_{B_n}(\bL) \oplus \End_{B_n}(\bF)
\]
and thus $\dim \End_{B_n}(\bE) = 2$. Since the identity operator
$1=\id_\bE$ evidently commutes with the action of $B_n$, we only need
a second, linearly independent, commuting operator.

We claim that $P_0$ commutes with the image of $B_n$; that is, $P_0$
commutes with $\beta_i$ for all $i = 1, \dots, n-1$. To see this,
write $v = v_0 + v'_0$ for $v_0 \in \bL$ and $v'_0 \in
\bF=\bL^\perp$. Then $\beta_i$ acts as $q_1$ on $v_0$ by Lemma
\ref{lem:eigen} and $\beta_i(v'_0) \in \bF$, so
\begin{gather*}
\beta_i P_0(v) = \beta_i(v_0) = q_1 v_0 \\ P_0 \beta_i(v) =
P_0(\beta_i(v_0) + \beta_i(v'_0)) = P_0(q_1 v_0) = q_1v_0.
\end{gather*}
This proves the claim. As $P_0$ is clearly not a scalar multiple of
the identity, we see that $\End_{B_n}(\bE) \cong \C \id_\bE \oplus \C
P_0$.

To compute the matrix of $P_0$ with respect to the standard basis $\ee_1,
\dots, \ee_n$ we calculate
\[
  \frac{ \bil{\ee_j}{\ee_1+\cdots+\ee_n} }{
    \bil{\ee_1+\cdots+\ee_n}{\ee_1+\cdots+\ee_n} } =
  \frac{q^{j-1}}{1+q+\cdots + q^{n-1}} .
\]
This shows that the matrix of $P_0$ is $(1+q+\cdots + q^{n-1})^{-1}$
times the matrix $P = (q^{j-1})_{1 \le i,j \le n}$, which completes
the proof.
\end{proof}

\begin{rmk}\label{rmk:P_0}
(i) Under the hypothesis of the theorem, the operator $1-P_0$ (where
$1=\id_\bE$) is projection onto $\bF$. It also commutes with the
action of $B_n$. Hence the set $\{1-P_0, P_0\}$ is another basis of
$\End_{B_n}(\bE)$.

(ii) The matrix $P$ factors as $P = UJ$ where $U$ is the
matrix of all ones and $J$ the matrix of the form
$\bil{-}{-}$. Also, $[n]_q = \text{trace}(P) = \text{trace}(J)$. 
\end{rmk}

\section{Schur--Weyl duality for $\bF^{\otimes k}$}\noindent
As always, we assume that $q_1q_2 \ne 0$. Assume that $q= -q_2/q_1$ is
not a root of unity. In this section, we establish a new instance of
Schur--Weyl duality for the tensor power $\bF^{\otimes k}$, regarded
as a representation of the group $B_n \times S_k$, where $S_k$ is the
symmetric group on $k$ letters acting by place-permutation and $B_n$
acts diagonally. In particular, we show that $\End_{B_n}(\bF^{\otimes
  k})$ is spanned by the image of the action of $S_k$.

\begin{lem}
Put $G = \rho(B_n)$, where $\rho: B_n \to \GL(\bF)$ is the reduced
Burau representation. Let $\ov{G}$ be the Zariski closure of $G$ in
$\GL(\bF)$. If $\ov{G}$ contains $\SL(\bF)$ then
\begin{equation*}
\End_{B_n}(\bF^{\otimes k}) = \End_{\GL(\bF)}(\bF^{\otimes k}) .
\end{equation*}
\end{lem}

\begin{proof}
First we observe that $\End_{\ov{G}}(\bF) = \End_{G}(\bF)$. Indeed, if
$A \in \End_{G}(\bF)$ then its commutant algebra
\[
\text{Comm}(A) = \{X \in \End(\bF): AX=XA\}
\]
is the solution set of the linear system $AX-XA=0$, so
$\text{Comm}(A)$ is a Zariski-closed subset of $\End(\bF)$ and thus
$\text{Comm}(A) \cap \GL(\bF)$ is Zariski-closed in $\GL(\bF)$. Now $G
\subset \text{Comm}(A) \cap \GL(\bF)$, so $\ov{G} \subset
\text{Comm}(A) \cap \GL(\bF)$, and hence $A
\in \End_{\ov{G}}(\bF)$. This proves that $\End_G(\bF)
\subset \End_{\ov{G}}(\bF)$, and thus (the opposite inclusion being
obvious) we have the desired equality
\[
\End_{\ov{G}}(\bF) = \End_{G}(\bF). 
\]
Now by hypothesis $\SL(\bF) \subset \ov{G} \subset \GL(\bF)$. Thus we
have
\[
\End_{\SL(\bF)}(\bF^{\otimes k}) \supset \End_{\ov{G}}(\bF^{\otimes k})
\supset \End_{\GL(\bF)}(\bF^{\otimes k}) .
\]
Now $\End_{\SL(\bF)}(\bF^{\otimes k}) = \End_{\GL(\bF)}(\bF^{\otimes
  k})$ since $\GL(\bF)$ and $\SL(\bF)$ differ only by scalars, so it
follows that all of the inclusions displayed above are equalities.
Since $\End_{B_n}(\bF^{\otimes k}) = \End_{G}(\bF^{\otimes k})$
holds by definition, we are done.
\end{proof}

The following is the key observation. 

\begin{thm}\label{thm:density}
Assume that $q_1q_2 \ne 0$ and that $q = -q_2/q_1$ is not a root of
unity. Then the Zariski closure (in $\GL(\bF)$) of the image of the
reduced Burau representation $\rho: B_n \to \GL(\bF)$ contains
$\SL(\bF)$. Hence
\[
\End_{B_n}(\bF^{\otimes k}) = \End_{\GL(\bF)}(\bF^{\otimes k}).
\]
\end{thm}

Although elementary, the proof of Theorem \ref{thm:density} is rather
technical, so we defer it to Appendix \ref{sec:proof}. Here we explore
consequences of the theorem.

We write $\lambda \vdash k$ to indicate that $\lambda$ is a partition
of $k$, meaning that $\lambda = (\lambda_1, \dots,
\lambda_{\ell(\lambda)})$ is a tuple of positive integers such that
$\lambda_1 \ge \lambda_2 \ge \cdots \ge \lambda_{\ell(\lambda)}$ and
$\lambda_1 + \cdots + \lambda_{\ell(\lambda)} = k$. The number
$\ell(\lambda)$ of parts of $\lambda$ is its length.

In his tome on the classical groups, Weyl defined the \emph{enveloping
  algebra} of a group representation $G \to \GL(U)$ to be the
subalgebra of $\End(U)$ generated by the image of the
representation. This coincides with the image of the group algebra $\C
G$ under the natural linear extension $\C G \to \End(U)$ of the
representation.

\begin{thm}\label{thm:SWD-R^k}
  Assume that $q_1q_2 \ne 0$ and $q = -q_2/q_1$ is not a root of
  unity.  Then the space $\bF^{\otimes k}$, regarded as a group
  representation of the direct product $B_n \times S_k$, satisfies
  Schur--Weyl duality, in the sense that the enveloping algebra of
  each group equals the centralizer of the other. In particular,
  \[
  \bF^{\otimes k} \cong \bigoplus_{\lambda \vdash k:\, \ell(\lambda)
    \le n-1} \Delta(\lambda) \otimes \Specht^\lambda
  \]
  where $\Delta(\lambda)$ is the Schur module for $\GL(\bF)$ of
  highest weight $\lambda$, regarded as $\C B_n$-module via the
  representation $B_n \to \GL(\bF)$, and $\Specht^\lambda$ is the
  Specht module for the symmetric group $S_k$ indexed by $\lambda$.
\end{thm}

\begin{proof}
As we have $\End_{B_n}(\bF^{\otimes k}) = \End_{\GL(\bF)}(\bF^{\otimes
  k})$ by Theorem \ref{thm:density}, the first claim follows from
classical Schur--Weyl duality (see e.g.,
\cites{Weyl,KP,Procesi,GW:09}), which goes back to Schur
\cite{Schur}. The decomposition in the last part then follows by well
known standard arguments in the theory of semisimple algebras.
\end{proof}

\begin{rmk}\label{rmk:combinatorics}
Let $\tabT$ be a tableau of shape $\lambda$, where $\lambda \vdash k$.
Let $R(\tabT)$, $C(\tabT)$ be the subgroups of $S_k$ preserving
respectively the numbers in the rows, columns of $\tabT$. In the group
algebra $\C S_k$, define
\begin{equation*}
  a_\tabT = \sum_{\sigma \in R(\tabT)} \sigma, \qquad
  b_\tabT = \sum_{\sigma \in C(\tabT)} \text{sgn}(\sigma) \sigma .
\end{equation*}
These elements of $\C S_k$, and the products $a_\tabT b_\tabT$,
$b_\tabT a_\tabT$ in either order, are \emph{Young symmetrizers}. Then
(see e.g.~\cites{Fulton, Barcelo-Ram}) we have isomorphisms:
\begin{enumerate}\renewcommand\labelenumi{(\roman{enumi})}
\item $\Delta(\lambda) \cong b_\tabT a_\tabT(\bF^{\otimes k})$.

\item $\Specht^\lambda \cong \C S_k b_\tabT a_\tabT$.
\end{enumerate}
It makes no difference here if $b_\tabT a_\tabT$ is replaced by
$a_\tabT b_\tabT$. The isomorphism in (i) is as $\C \GL(\bF)$-modules,
which induces an isomorphism as $\C B_n$-modules. Furthermore,
$\Delta(\lambda)$ (resp., $\Specht^\lambda$) has a basis indexed by
the set of all semistandard (resp., standard) tableaux of shape
$\lambda$.
\end{rmk}

\section{The partition algebra}\label{sec:Ptn-alg}\noindent
Fix a complex number $z \ne 0$. The partition algebra $\Ptn_r(z)$ was
introduced independently by Martin and Jones
\cites{Martin:book,Martin:94,Jones:94} in connection with the Potts
model in mathematical physics, as a generalization of the
Temperley--Lieb \cite{Temperley-Lieb} and Brauer \cite{Brauer}
algebras. It is a ``diagram algebra'' with a graphical basis
consisting of graphs on $2r$ nodes, depending on a parameter $z$. We
refer the reader to \cite{HR:05} for a convenient summary of many
basic properties of $\Ptn_r(z)$.

Recall that $\Ptn_r(z)$ has a basis $\Ptn_r$ consisting of all set
partitions of the set $\{1,\dots r, 1', \dots, r'\}$. Equivalently,
$\Ptn_r$ may be identified with the set of equivalence relations on
$\{1,\dots r, 1', \dots, r'\}$. The various (disjoint) subsets of a
set partition $d$ in $\Ptn_r$ are called the \emph{blocks} of $d$.
Conventionally, $d$ is depicted as a graph with $2r$ nodes, arranged
in two parallel horizontal rows, numbered $1, \dots, r$ on the top and
$1', \dots, r'$ on the bottom, with two nodes connected by a path if
and only if they lie in the same block. Thus the connected components
of the graph determine the blocks of the set partition (and the
graphical depiction is not necessarily unique).

Composition of graphs $d_1,d_2$ in $\Ptn_r$ is defined by stacking
$d_1$ above $d_2$, identifying the top row of $d_2$ with the bottom
row of $d_1$, and omitting any connected components contained entirely
in the middle two (identified) rows. The result is always another
graph (set partition in $\Ptn_r$) that we denote by $d_1 \circ d_2$.
Multiplication $d_1d_2$ in the partition algebra $\Ptn_r(z)$ is then
defined by setting
\begin{equation}\label{eq:ptn-alg-mult}
d_1d_2 = z^N \, (d_1 \circ d_2)
\end{equation}
where $N =$ the number of omitted connected components. The linear
extension of this rule defines an associative multiplication on
$\Ptn_r(z)$.

The composition rule $(d_1,d_2) \mapsto d_1 \circ d_2$ makes $\Ptn_r$
into a monoid, called the \emph{partition monoid}. The specialization
$\Ptn_r(1)$ is isomorphic to the semigroup algebra $\C \Ptn_r$,
consisting of all formal $\C$-linear combinations of the elements of
$\Ptn_r$, with multiplication in $\C \Ptn_r$ the linear extension of
the monoid operation on $\Ptn_r$.

The only fact about $\Ptn_r(z)$ that we need is the following.

\begin{lem}[\cite{HR:05}]\label{lem:gens-Ptn}
The partition algebra $\Ptn_r(z)$ is generated by the diagrams of the
form
\begin{gather*}
s_i =\;
\begin{minipage}{3.9cm}
\begin{tikzpicture}[scale=.5]
  \filldraw[fill= black!12,draw=black!12,line width=4pt] (1,0) rectangle (8,1);
    \foreach \x in {1,3,4,5,6,8} {
      \filldraw [black] (\x, 0) circle (2pt);
      \filldraw [black] (\x, 1) circle (2pt);
    }
    \draw (1,1)--(1,0) (3,1)--(3,0);
    \draw (6,1)--(6,0) (8,1)--(8,0);
    \draw (4,1)--(5,0) (5,1)--(4,0);
    \node at (2,0.5) {$\cdots$};
    \node at (7,0.5) {$\cdots$};
\end{tikzpicture}
\end{minipage} , \qquad
p_j =\;
\begin{minipage}{3.5cm}
\begin{tikzpicture}[scale=.5]
  \filldraw[fill= black!12,draw=black!12,line width=4pt] (1,0) rectangle (7,1);
    \foreach \x in {1,3,4,5,7} {
      \filldraw [black] (\x, 0) circle (2pt);
      \filldraw [black] (\x, 1) circle (2pt);
    }
    \draw (1,1)--(1,0) (3,1)--(3,0);
    \draw (5,1)--(5,0) (7,1)--(7,0);
    \node at (2,0.5) {$\cdots$};
    \node at (6,0.5) {$\cdots$};
\end{tikzpicture}
\end{minipage},  \\
\intertext{and } p_{i+\frac{1}{2}} =\;
\begin{minipage}{4cm}
\begin{tikzpicture}[scale=.5]
  \filldraw[fill= black!12,draw=black!12,line width=4pt] (1,0) rectangle (8,1);
  \foreach \x in {1,3,4,5,6,8} {
      \filldraw [black] (\x, 0) circle (2pt);
      \filldraw [black] (\x, 1) circle (2pt);
    }
    \draw (1,1)--(1,0) (3,1)--(3,0);
    \draw (6,1)--(6,0) (8,1)--(8,0);
    \draw (4,1)--(4,0) (5,1)--(5,0);
    \draw (4,1)--(5,1) (4,0)--(5,0);
    \node at (2,0.5) {$\cdots$};
    \node at (7,0.5) {$\cdots$};
\end{tikzpicture}
\end{minipage}
\end{gather*}
for $i = 1, \dots, r-1$ and $j = 1, \dots, r$. In $p_j$ column $j$ is
the unusual one, while in $s_i$ and $p_{i+\frac{1}{2}}$ the same is
true of columns $i,i+1$.
\end{lem}

\begin{rmk}\label{rmk:ptn}
(i) See \cite{HR:05}*{Thm.~1.11} for a set of defining relations on
  the above generators that determine the partition algebra.

(ii) It is known \cites{Martin:94,Martin:96} that the set
  $\Lambda_{\le r} = \{\lambda \vdash k: k = 0, 1, \dots, r\}$
  indexes the isomorphism classes of irreducible $\Ptn_r(z)$-modules.
\end{rmk}

\section{Example}\label{sec:example}\noindent
Assume that $q_1q_2 \ne 0$ and that $q = -q_2/q_1$ is not a root of
unity. Our next overarching goal is to compute the centralizer algebra
$\End_{B_n}(\bE^{\otimes r})$ of the tensor powers of the unreduced
Burau representation $\bE$. In this section we work through a
motivating example (the case $r=2$) that suggests the general result
to follow.

Thanks to our assumptions, we know that $\bE = \bL \oplus \bF$ as $\C
B_n$-modules. Thus we have
\begin{equation}\label{eq:VoV}
\begin{aligned}
\bE \otimes \bE &= (\bL \oplus \bF) \otimes (\bL \oplus \bF)\\ &\cong
(\bL \otimes \bL) \oplus (\bL \otimes \bF \oplus \bF \otimes \bL)
\oplus (\bF \otimes \bF)
\end{aligned}
\end{equation}
as $\C B_n$-modules. From the standard isomorphism $X \otimes Y \cong
Y \otimes X$ for group representations $X,Y$ we deduce that $\bL
\otimes \bF \cong \bF \otimes \bL$. From Theorem \ref{thm:SWD-R^k}
we have the semisimple decomposition as $\C B_n$-modules
\begin{equation}\label{eq:RoR}
  \bF \otimes \bF \cong
  \begin{cases}
    \Sym^2 \bF \oplus \wedge^2 \bF & \text{ if } n > 2 \\
    \text{an irreducible module} & \text{ if } n = 2 \\
  \end{cases}
  \end{equation}
since $\bF \otimes \bF$ is one-dimensional, hence irreducible, when
$n=2$. We note that $\Sym^2 \bF \cong \Delta(2)$ and $\wedge^2
\bF \cong \Delta(1^2)$ as $\C \GL(\bF)$-modules, and hence also as $\C
B_n$-modules.

It is easy to check that $\bL \otimes \bL$, $\bL \otimes \bF$, and the
irreducible components on the right hand side of \eqref{eq:RoR} are
pairwise non-isomorphic modules. Thus by Schur's Lemma it follows that
for $n>2$ we have
\begin{equation}
  \begin{aligned}
  \End_{B_n}(\bE \otimes \bE) \cong
  & \End_{B_n}(\bL\otimes \bL) \oplus \End_{B_n}(\bL \otimes \bF
  \oplus \bF \otimes \bL)\\ &\oplus \End_{B_n}(\Sym^2 \bF)
  \oplus \End_{B_n}(\wedge^2 \bF)
  \end{aligned}
\end{equation}
while for $n=2$ we have
\begin{equation}
  \begin{aligned}
  \End_{B_n}(\bE \otimes \bE) \cong
  & \End_{B_n}(\bL\otimes \bL) \oplus \End_{B_n}(\bL \otimes \bF
  \oplus \bF \otimes \bL) \\ & \oplus \End_{B_n}(\bF \otimes \bF).
  \end{aligned}
\end{equation}
In light of the isomorphism $\bL \otimes \bF \cong \bF \otimes \bL$
already mentioned, it follows by Schur's Lemma that
\begin{equation}
  \dim \End_{B_n}(\bE \otimes \bE) =
  \begin{cases}
    1^2 + 2^2 + 1^2 + 1^2 = 7 & \text{ if } n>2 \\
    1^2 + 2^2 + 1^2  = 6 & \text{ if } n=2.
  \end{cases}
\end{equation}

Now we consider the above computation from a different perspective,
which sheds much light on the general situation. Let $p$ be the
pseudo-projection $\bE \twoheadrightarrow \bL$ defined by the matrix
$P = (q^{j-1})_{1 \le i,j \le n}$ of Theorem \ref{thm:P}, with respect
to the standard basis $\{\ee_1, \dots, \ee_n\}$ of $\bE$. Then the linear
operators
\begin{equation}
  p_1 = p \otimes \id_\bE , \quad p_2 =  \id_\bE \otimes p
\end{equation}
are both elements of $\End_{B_n}(\bE \otimes \bE)$. Since the diagonal
action of $B_n$ evidently commutes with place-permutations, the swap
operator $s: \bE \otimes \bE \to \bE \otimes \bE$ defined by $s(v_1
\otimes v_2) = v_2 \otimes v_1$ also belongs to $\End_{B_n}(\bE
\otimes \bE)$. Let $A$ be the subalgebra of $\End(\bE \otimes
\bE)$ generated by $p_1, p_2, s$. Since the generators commute with
the action of $B_n$, we have an inclusion $A
\subseteq \End_{B_n}(\bE \otimes \bE)$.  Now the
relations
\begin{equation}\label{eq:def-relns}
\begin{aligned}
  s^2 &= 1, \quad p_j^2 = [n]_q \,p_j \quad (j=1,2), \quad p_1p_2 =
  p_2p_1 \\ p_2 &= sp_1s, \quad sp_1p_2 = p_1p_2s = p_1p_2
\end{aligned}
\end{equation}
are easily checked to hold in $A$, where $1 = \id_{\bE \otimes
  \bE}$. These relations imply that $A$ is spanned over $\C$ by the
seven endomorphisms
\begin{equation}\label{eq:span-set}
  \{ 1, s, p_1, p_2, sp_1, p_1s, p_1p_2 \}
\end{equation}
since any other word in the generators reduces to one of these. 

\begin{lem}
  Assume that $q= -q_2/q_1$ is not a root of unity.
  \begin{enumerate}
  \item If $n=2$ then the elements in \eqref{eq:span-set} satisfy the
    linear dependence relation
    \[
    p_1 - sp_1 - p_1s + p_2 -(1+q)(1 - s) = 0
    \]
    and $\dim A = 6$.
  \item If $n>2$ then the spanning elements in \eqref{eq:span-set} are
    linearly independent, hence a basis of $A$, and thus $\dim
    A = 7$. Thus $A$ is the algebra defined by generators
    $s,p_1,p_2$ subject to the defining relations
    \eqref{eq:def-relns}.
  \end{enumerate}
  In both cases $A = \End_{B_n}(\bE \otimes \bE)$. 
\end{lem}

\begin{proof}
First we note that the operators in \eqref{eq:span-set} take $\ee_i
\otimes \ee_j$ to
\[
\ee_i \otimes \ee_j, \ee_j \otimes \ee_i, q^{i-1 } \ff_0 \otimes \ee_j, q^{j-1}
\ee_i \otimes \ff_0, q^{i-1} \ee_j \otimes \ff_0, q^{j-1} \ff_0 \otimes \ee_i,
q^{(i-1)(j-1)} \ff_0 \otimes \ff_0
\]
respectively.  Assume that scalars $x_j \in \C$ exist such that
\[
x_1 \,1 + x_2\, s + x_3\, p_1 + x_4\, p_2 + x_5\, sp_1 + x_6\,
p_1s + x_7\, p_1p_2 = 0.
\]
Then the mapping on the left hand side takes $\ee_i \otimes \ee_j$ to
zero, for all $i,j = 1, \dots, n$. In particular, by taking $i=j$ we
obtain
\[
(x_1+x_2) \ee_i \otimes \ee_i + (x_3+x_6)q^{i-1} \ff_0 \otimes \ee_i +
(x_4+x_5) q^{i-1} \ee_i \otimes \ff_0 + x_7 q^{(i-1)^2} \ff_0 \otimes \ff_0 =
0
\]
where $\ff_0 = \ee_1+\cdots + \ee_n$. This holds for all $i$, so we conclude
that
\[
x_2 = -x_1, \quad x_6 = -x_3, \quad x_5 = -x_4, \quad x_7 = 0.
\]
So we are reduced to finding $x_1, x_3, x_4$. Putting this information
back into the first equation and setting $i=1$, $j=2$ gives
\[
x_1 \ee_1 \otimes \ee_2 - x_1 \ee_2 \otimes \ee_1 + x_3 \ff_0 \otimes \ee_2 + x_4
q \ee_1 \otimes \ff_0 - x_4\ee_2 \otimes \ff_0 - x_3q \ff_0 \otimes \ee_1 = 0.
\]
Looking at the coefficients of $\ee_1\otimes \ee_2$, $\ee_2 \otimes \ee_1$
respectively yields the equations
\[
x_1+x_3+x_4q = 0, \quad -x_1-x_4-x_3q = 0.
\]
If $n=2$ then these equations have the non-trivial solution
$x_3=x_4=1$, $x_1 = -(1+q) = -[2]_q$. This yields the dependence
relation that proves (a).  If $n>2$ then there are additional
equations coming from the coefficients of $\ee_1 \otimes \ee_3$, $\ee_3
\otimes \ee_1$ which force $x_3=x_4=x_1 =0$, proving the linear
independence claim in (b). Once we know that the seven spanning
elements in \eqref{eq:span-set} are linearly independent, it follows
that $\dim A = 7$ and thus the inclusion $A
\subseteq \End_{B_n}(\bE \otimes \bE)$ must be an equality, by
dimension comparison. This gives the second claim in (b), since the
algebra defined by the generators and relations is at most
seven-dimensional.

The last claim is now clear as well, since in the $n=2$ case it is
easy to check that the six spanning elements in \eqref{eq:span-set}
are linearly independent, once a chosen dependent operator has been
eliminated.
\end{proof}

Recall that the Weyl group $W_n \subset \GL(\bE)$ acts diagonally on
tensor space $\bE^{\otimes r}$, by restriction of the diagonal action
of $\GL(\bE)$.  In this situation, there is a natural commuting action
of the partition algebra $\Ptn_r(n)$ on $r$ nodes with parameter $n$,
and Schur--Weyl duality holds for $\bE^{\otimes r}$ regarded as a $(\C
W_n, \Ptn_r(n))$-bimodule. In particular, the centralizer algebra
$\End_{W_n}(\bE^{\otimes r})$ is the image of the representation
$\Ptn_r(n) \to \End(\bE^{\otimes r})$.

Our final result of this section makes a connection with the partition
algebra. To make that connection, we replace the parameter $n$ by its
$q$-analogue $[n]_q$.

\begin{prop}
The generic seven-dimensional algebra defined by generators $p_1, p_2,
s$ subject to the relations \eqref{eq:def-relns} of this section may
be identified with the subalgebra $\PP_2([n]_q)$ of the partition
algebra $\Ptn_2([n_q])$ with parameter $[n]_q$ spanned by all
partition diagrams on two nodes with no horizontal edges.
\end{prop}

\begin{proof}
This follows from the presentation \cite{HR:05}*{Thm.~1.11} of the
partition algebra along with results of this section.  The generators
of $\Ptn_r([n_q])$ are denoted by $p_i$, $s_j$, $p_{i+1/2}$ in
\cite{HR:05}, for $i=1,\dots, r$ and $j = 1, \dots, r-1$. Let
$\mathcal{S}$ be the subalgebra generated by the $p_i, s_j$.  In the
case $r=2$ one can check that the defining relations on our
$p_1,p_2,s$ are satisfied in $\Ptn_2([n]_q)$, where we identify our
$p_i$ with those elements of \cite{HR:05} and identify $s = s_1$. Thus
our algebra is isomorphic to a quotient of the corresponding
subalgebra $\PP_2([n]_q)$. By dimension comparison, the two algebras
are equal.
\end{proof}

To complete the connection to the partition algebra, we note that the
basis elements in the algebra $\PP_2([n]_q)$ can be written in terms
of generators as
\[
  1, s, p_1, p_2, p_1s, sp_1, p_1p_2
\]
and these elements correspond respectively with the seven partition
diagrams listed below:
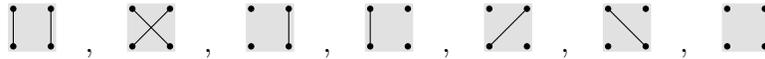
\[
\begin{tikzpicture}[scale=.5]
  \filldraw[fill= black!12,draw=black!12,line width=4pt] (1,0) rectangle (2,1);
  \foreach \x in {1,...,2} {
      \filldraw [black] (\x, 0) circle (2pt);
      \filldraw [black] (\x, 1) circle (2pt);
    }
    \draw (1,1)--(1,0) (2,1)--(2,0);
\end{tikzpicture} \quad , \quad
\begin{tikzpicture}[scale=.5]
  \filldraw[fill= black!12,draw=black!12,line width=4pt] (1,0) rectangle (2,1);
    \foreach \x in {1,...,2} {
      \filldraw [black] (\x, 0) circle (2pt);
      \filldraw [black] (\x, 1) circle (2pt);
	 }
    \draw (1,1)--(2,0) (2,1)--(1,0);
\end{tikzpicture} \quad , \quad
\begin{tikzpicture}[scale=.5]
  \filldraw[fill= black!12,draw=black!12,line width=4pt] (1,0) rectangle (2,1);
    \foreach \x in {1,...,2} {
      \filldraw [black] (\x, 0) circle (2pt);
      \filldraw [black] (\x, 1) circle (2pt);
	 }
    \draw (2,1)--(2,0);
\end{tikzpicture} \quad , \quad
\begin{tikzpicture}[scale=.5]
  \filldraw[fill= black!12,draw=black!12,line width=4pt] (1,0) rectangle (2,1);
    \foreach \x in {1,...,2} {
      \filldraw [black] (\x, 0) circle (2pt);
      \filldraw [black] (\x, 1) circle (2pt);
	 }
    \draw (1,1)--(1,0);
\end{tikzpicture} \quad , \quad
\begin{tikzpicture}[scale=.5]
  \filldraw[fill= black!12,draw=black!12,line width=4pt] (1,0) rectangle (2,1);
    \foreach \x in {1,...,2} {
      \filldraw [black] (\x, 0) circle (2pt);
      \filldraw [black] (\x, 1) circle (2pt);
	 }
    \draw (2,1)--(1,0);
\end{tikzpicture} \quad , \quad
\begin{tikzpicture}[scale=.5]
  \filldraw[fill= black!12,draw=black!12,line width=4pt] (1,0) rectangle (2,1);
    \foreach \x in {1,...,2} {
      \filldraw [black] (\x, 0) circle (2pt);
      \filldraw [black] (\x, 1) circle (2pt);
	 }
    \draw (1,1)--(2,0);
\end{tikzpicture} \quad , \quad
\begin{tikzpicture}[scale=.5]
  \filldraw[fill= black!12,draw=black!12,line width=4pt] (1,0) rectangle (2,1);
    \foreach \x in {1,...,2} {
      \filldraw [black] (\x, 0) circle (2pt);
      \filldraw [black] (\x, 1) circle (2pt);
	 }
\end{tikzpicture} 
\]
These are precisely the partition diagrams in $\Ptn_2([n]_q)$ which
depict partial permutations.

\section{The algebra $\PP_r(z)$ and its representations}
\label{sec:P'}\noindent
Now we study the subalgebra $\PP_r(z)$ of $\Ptn_r(z)$ spanned by all
graphs satisfying the conditions of Definition \ref{def:part-perm}(a)
below.  The main results of this section are:
\begin{enumerate}\renewcommand{\labelenumi}{(\roman{enumi})}
  
\item A description of $\PP_r(z)$ by generators and relations.
\item A proof that $\PP_r(z)$ is semisimple if $z \ne 0$.
  
\item A construction of the irreducible representations $\{
  C(\lambda): \lambda \in \Lambda \}$ of $\PP_r(z)$, where
  $\Lambda$ is the union of the set of partitions of $k$,
  as $k$ runs from 0 to $r$.
  
\end{enumerate}
We will also see that $\PP_r(z)$ has a basis in natural bijection with
the set of partial permutations (bijections between subsets) on $\{1,
\dots, r\}$.

\begin{defn}\label{def:part-perm}
We need the following notation. 
\begin{enumerate}
\item Let $\PP_r$ be the subset of $\Ptn_r$ consisting of all $d$ in
  $\Ptn_r$ satisfying the properties:
  \begin{itemize}
  \item every block of $d$ has cardinality 1 or 2.
  \item every block of cardinality 2 contains exactly one element from
    $\{1, \dots, r\}$ and one from $\{1', \dots, r'\}$.
  \end{itemize}
\item Every $d \in \PP_r$ determines a unique bijection
  $\text{bi}(d)$ between two subsets of $\{1, \dots, r\}$, sending $i$
  to $j$ if and only if $\{i,j'\}$ is a block of $d$ of cardinality
  2. Such bijections are \emph{partial permutations} on $\{1, \dots,
  r\}$.
\item Let $\PP_r(z)$ be the subalgebra of $\Ptn_r(z)$ with basis $\PP_r$.
\end{enumerate}
\end{defn}

Composition $(d_1, d_2) \mapsto d_1 \circ d_2$ makes $\PP_r$ into a
submonoid of the partition monoid $\Ptn_r$.  Setting $f_i =
\text{bi}(d_i)$ for $i=1,2$ we have the composite map $f_1f_2$ defined
by
\[
  x(f_1 f_2) = (xf_1) f_2
\]
where we write maps on the \emph{right} of their arguments (for
compatibility with diagrammatic composition). Then $\text{bi}(d_1
\circ d_2) = f_1f_2$, so $\text{bi}$ is a monoid isomorphism of
$\PP_r$ onto the monoid of all partial permutations of $\{1, \dots,
r\}$ (bijections between subsets of $\{1, \dots, r\}$) under
functional composition.  In other words, $\PP_r$ is isomorphic to the
``symmetric inverse semigroup'' studied by Munn
\cites{Munn:57a,Munn:57b}, Solomon \cite{Solomon}, and others.

Solomon observed that $\PP_r$ is isomorphic to the monoid of $r\times
r$ matrices (under matrix multiplication) with at most one non-zero
entry, equal to 1, in each row and column. Such matrices enumerate the
ways to place non-attacking rooks on an $r \times r$ chessboard, so
Solomon introduced the term ``rook monoid'' for $\PP_r$.  Observe
that the algebra $\PP_r(1)$ is the semigroup algebra $\C \PP$ of the
rook monoid.

\begin{defn}
  If $d \in \PP_r$ we define its \emph{rank}, written $\rank(d)$, to
  be the number of edges in $d$. (This coincides with the rank of its
  rook matrix realization, in the usual sense of rank of a matrix.)
\end{defn}

There is a unique element of $\PP_r$ of rank zero. Elements of $\PP_r$
of maximum rank $r$ correspond to permutations of $\{1, \dots, r\}$
under the map $d \mapsto \text{bi}(d)$, so the symmetric group $S_r$
is isomorphic to a subgroup of $\PP_r$. Furthermore,
\begin{equation}
\dim \PP_r(z) = |\PP_r| = \sum_{k=0}^r \binom{r}{k}^2 k!
\end{equation}
as there are $\binom{r}{k}^2 k!$ partial permutations
of rank $k$, for each $k = 0,1, \dots, r$.

Henceforth we will identify $d \in \PP_r$ with its underlying
bijection $\text{bi}(d)$, thus identifying $\PP_r$ with the rook
monoid of partial permutations.

\begin{lem}
The rook monoid $\PP_r$ is isomorphic to the submonoid of $\Ptn_r$
generated by the diagrams
\[
s_i =\;
\begin{minipage}{4cm}
\begin{tikzpicture}[scale=.5]
  \filldraw[fill= black!12,draw=black!12,line width=4pt] (1,0) rectangle (8,1);
    \foreach \x in {1,3,4,5,6,8} {
      \filldraw [black] (\x, 0) circle (2pt);
      \filldraw [black] (\x, 1) circle (2pt);
    }
    \draw (1,1)--(1,0) (3,1)--(3,0);
    \draw (6,1)--(6,0) (8,1)--(8,0);
    \draw (4,1)--(5,0) (5,1)--(4,0);
    \node at (2,0.5) {$\cdots$};
    \node at (7,0.5) {$\cdots$};
\end{tikzpicture}
\end{minipage}
\quad \text{and} \quad
p_j =\;
\begin{minipage}{4cm}
\begin{tikzpicture}[scale=.5]
  \filldraw[fill= black!12,draw=black!12,line width=4pt] (1,0) rectangle (7,1);
  \foreach \x in {1,3,4,5,7} {
      \filldraw [black] (\x, 0) circle (2pt);
      \filldraw [black] (\x, 1) circle (2pt);
    }
    \draw (1,1)--(1,0) (3,1)--(3,0);
    \draw (5,1)--(5,0) (7,1)--(7,0);
    \node at (2,0.5) {$\cdots$};
    \node at (6,0.5) {$\cdots$};
\end{tikzpicture}
\end{minipage}
\]
for $i = 1, \dots, r-1$, $j=1,\dots, r$, where the crossing in $s_i$
is in columns $i$, $i+1$ and column $j$ is the only column in $p_j$
without an edge. Hence $\PP_r(z)$ is the subalgebra of $\Ptn_r(z)$
generated by the same diagrams. 
\end{lem}

\begin{proof}
The $s_i$ generate a copy of the symmetric group $S_r$. This is the
set of diagrams in $\PP_r(z)$ of full rank $r$. The diagram $d_k = p_1
\cdots p_{k-1} p_k$ is a diagram of rank $r-k$. Left or right
multiplication by a permutation diagram in $S_r$ acts on $d_k$ to
permute the nodes in the top or bottom row, respectively, thus
generating all diagrams of rank $r-k$, for each $k = 0,1,\dots, r$.
\end{proof}

Let $\J$ be the set of diagrams of the form $\prod_{j \in J} p_j$,
such that $J$ is a subset of $\{1, \dots, r\}$. If $i \ne j$ then
$p_i, p_j$ commute, so the order of the factors in the product is
irrelevant.

\begin{lem}\label{lem:extn}
  Any diagram $d$ in $\PP_r$ may be written in the form
  \[d = pw = wp'\] for some $w \in S_r$, some $p,p' \in \J$.
\end{lem}

\begin{proof}
Given $d$, we may extend $d$ to a permutation $w$ in $S_r$ by adding
edges (usually in more than one way). Pre-multiplying $w$ by an
appropriate diagram $p \in \J$ erases those added edges; similarly for
post-multiplication.
\end{proof}

The extension of $d$ to a permutation $w$ in Lemma \ref{lem:extn} is
in general far from unique. Our next task is to find a way to describe
all such extensions.  Suppose that $d \in \PP_r$. Then $d$ is a
bijection mapping a subset $J \subset \{1, \dots, r\}$ onto another
subset, and there exists a unique map $d^+: \{1,\dots,r\} \to
\{0,1,\dots, r\}$ such that
\[
x (d^+) = 
\begin{cases}
x d & \text{ if } x \in J \\
0 & \text{ otherwise}.
\end{cases}
\]
For instance, if $r=8$ the partial permutation $d = \;
\begin{minipage}{3.8cm}
\begin{tikzpicture}[scale=.5]
  \filldraw[fill= black!12,draw=black!12,line width=4pt] (1,0) rectangle (8,1);
  \foreach \x in {1,...,8} {
      \filldraw [black] (\x, 0) circle (2pt);
      \filldraw [black] (\x, 1) circle (2pt);
    }
    \draw (1,1)--(2,0) (2,1)--(3,0);
    \draw (8,1)--(7,0) (7,1)--(6,0);
    \draw (4,1)--(5,0) (5,1)--(4,0);
\end{tikzpicture}
\end{minipage}$
defines the map $d^+$ given by
\[
d^+ =
\left(
\begin{array}{cccccccc}
  1 & 2 & 3 & 4 & 5 & 6 & 7 & 8\\
  2 & 3 & 0 & 5 & 4 & 0 & 6 & 7
\end{array}
\right)
\]
in the obvious extension of the usual two-line notation for
permutations.

A \emph{cycle} $(i_1, i_2, \dots i_m)$ is the map $i_1 \mapsto i_2
\mapsto \cdots \mapsto i_m \mapsto i_1$ as usual in the theory of
symmetric groups.  A \emph{link} $[j_1,j_2, \dots, j_m]$ is the map
$j_1 \mapsto j_2 \mapsto \cdots \mapsto j_m \mapsto 0$. In both cases,
all other indices are fixed.  Disjoint cycles or links commute.  Munn
\cite{Munn:57b} (cf.\ also \cite{Grood}) observed that every partial
permutation can be written as a product of disjoint cycles and links,
and this expression is unique apart from the order of the (commuting)
factors. For example, the partial permutation $d$ above can be
expressed in the cycle-link notation as
\[
d = [1,2,3](4,5)[8,7,6].
\]
By replacing all links in this expression by the corresponding cycle,
we obtain a permutation $w(d) \in S_r$ that we call the
\emph{canonical} extension of $d$. So in the above example, $w(d) =
(1,2,3)(4,5)(8,7,6)$. The cycles in $w(d)$ corresponding to links in
$d$ will be called \emph{link-cycles}. For instance, the link-cycles
for $d = [1,2,3](4,5)[8,7,6]$ are $(1,2,3)$, $(8,7,6)$.

\begin{lem}\label{lem:all-extns}
Let $d \in \PP_r$ be a partial permutation, regarded as a bijection
$d: X \to Y$ where $X,Y$ are subsets of $\{1,\dots,r\}$. Let $w(d)$ in
$S_r$ be the canonical extension of $d$ to a permutation.
\begin{enumerate}
\item Let $X' = \{1,\dots,r\} - X$ and $Y' = \{1,\dots,r\} - Y$. Then
  \[
  d = p w(d) = w(d) p',
  \]
  where $p = \prod_{j \in X'} p_j$, $p' = \prod_{j \in Y'} p_j$.
\item If $w \in S_r$ is any extension of $d$, then $w$ is a product of
  disjoint cycles obtained from the cycle-link decomposition of $d$ by
  joining one or more of its link-cycles in $w(d)$.
\end{enumerate}
\end{lem}

\begin{proof}
(a) is diagrammatically obvious. To prove (b), notice that in order to
  find an extension $w$ of $d$, one has to choose an isolated node in
  the bottom row of $d$ for each isolated node in the top row of $d$.
  In other words, pick a bijection from $X'$ onto $Y'$. The canonical
  choice preserves the cycles in $w(d)$; every other choice joins one
  or more link-cycles together. For example, if $d =
  [1,2,3](4,5)[8,7,6]$ then the two possible extensions of $d$ are
  $w_1=w(d) = (1,2,3)(4,5)(8,7,6)$ and $w_2 = (1,2,3,8,7,6)(4,5)$.
\end{proof}

\begin{thm}\label{thm:present}
The algebra $\PP_r(z)$ of partial permutations is isomorphic to the
algebra generated by $s_i, p_j$ for $i=1,\dots, r-1$ and $j = 1,
\dots, r$ subject to the relations
\begin{enumerate}
  \item $p_j^2 = z p_j$,  $p_i p_j = p_j p_i$ ($i \ne j$).
  \item $s_i^2 = 1$,  $s_is_{i+1}s_i = s_{i+1}s_is_{i+1}$, 
    $s_is_j = s_js_i$ if $|i-j|>1$.
  \item $s_ip_ip_{i+1} = p_ip_{i+1}s_i = p_ip_{i+1}$,  $s_ip_is_i
    = p_{i+1}$,  $s_ip_j = p_js_i$ if $j \ne i, i+1$.
\end{enumerate}
\end{thm}

\begin{proof}
That the generators satisfy these relations is easily checked by
direct computation with diagrams (see also the proof of Theorem 1.11
in \cite{HR:05}). It follows that $\PP_r(z)$ is a homomorphic image of
the algebra $A$ defined by the given generators and relations, so
$\dim A \ge \dim \PP_r(z)$.  To finish, it suffices to show that
$\dim A \le \dim \PP_r(z)$.

The defining relations for $A$ imply the following extension of the
first relation in (c):
\begin{equation}\label{eq:insertion}
(i,j) p_i p_j = p_i p_j (i,j) = p_i p_j \qquad (i \ne j)
\end{equation}
where $(i,j)$ is the transposition that swaps $i$ with $j$. This
follows by an easy argument that we leave to the reader.

Let $x$ be any word in $A$ in the generators. The relation $s_ip_is_i
= p_{i+1}$ in (c) is equivalent to
\[
s_ip_i = p_{i+1}s_i.
\]
Thus we may use the last two relations in (c) to commute all the $s_i$
to the right of all the $p_j$. Since the $p_j$ all commute with each
other, their order is immaterial.  If any $p_j$ is repeated we can use
the first relation in (a) to eliminate that repetition, at the expense
of introducing a scalar multiple. This shows that $A$ is spanned by
the set of words in the generators of the form
\[
p  w \quad \text{ where } \quad p = \textstyle \prod_{j\in X}
p_j,\quad w \in S_r
\]
and $X$ is a subset of $\{1,\dots,r\}$. Of course, $w$ is expressed by
some product of the $s_i$ (in more than one way) as usual in the
Coxeter presentation of symmetric groups.

To finish, it suffices to prove that the spanning set $\{ pw \}$ in
the previous paragraph can be reduced to one in bijection with the set
of partial permutations. Let $d$ be a partial permutation diagram. The
canonical word $p w(d)$ is one expression in $A$ for $d$. If $p w$ is
any other such word, then it is obtained from $pw(d)$ by inserting one
or more transpositions $(i,j)$ between $p$, $w$ according to rule
\eqref{eq:insertion}, hence $p w = p w(d)$ in $A$. Such insertions
correspond precisely to joining one or more link-cycles in $w(d)$,
according to Lemma \ref{lem:all-extns}(b). This completes the proof.
\end{proof}

Theorem \ref{thm:present} immediately implies that $\PP_r(z)$ is
isomorphic to the semigroup algebra $\C\PP_r$ of the rook monoid, for
any $z \ne 0$, a (perhaps surprising) fact which does not seem to have
been noticed before.

\begin{cor}\label{cor:present}
  For any $z \ne 0$, there is an algebra isomorphism $\PP_r(z) \cong
  \PP_r(1) = \C \PP_r$.
\end{cor}

\begin{proof}
The isomorphism is defined on generators by sending $s_i \mapsto s_i$
and $p_j \mapsto z^{-1} p_j$ for all $i = 1, \dots, r-1$, $j = 1, \dots,
r$.
\end{proof}

Next we wish to prove the semisimplicity of $\PP_r(z)$ and describe
its irreducible representations. Although the semisimplicity follows
from Corollary \ref{cor:present} and Corollary 2.19 of \cite{Solomon},
which goes all the way back to results of Munn in the 1950s, we prefer
to give a more efficient modern approach, exploiting the idea of
``iterated inflation'' introduced by K\"{o}nig and Xi \cites{KX:99,
  KX:01} into the theory of cellular algebras \cite{GL:96}. We will
follow the recent paper \cite{Green-Paget} which in particular gives a
useful criterion for verifying that a given algebra is an iterated
inflation of cellular algebras.

\begin{defn}\label{def:datum}
\renewcommand{\labelenumi}{(\roman{enumi})}
If $d$ in $\PP_r$ is a partial permutation, regarded as a bijection $d:
X \to Y$, where $X,Y$ are subsets of $\{1, \dots, r\}$, we write
\[
\dom(d) = X = \{x_1, \dots, x_k \},\quad \im(d) = Y = \{y_1, \dots,
y_k \}
\]
for the domain and image of $d$, where it is assumed that $x_1 < x_2 <
\cdots < x_k$, $y_1 <y_2 < \cdots < y_k$. The \emph{underlying
  permutation} of $d$ is the element $\pi(d)$ in $S_k$ defined by the
rule
\[
i \pi(d) = j \iff x_i d = y_j .
\]
Diagrammatically, this amounts to renumbering the nodes in the domain
and image of the corresponding graph.
\end{defn}

For example, the partial permutation $d = \;
\begin{minipage}{3.9cm}
\begin{tikzpicture}[scale=.5]
  \filldraw[fill= black!12,draw=black!12,line width=4pt] (1,0) rectangle (8,1);
  \foreach \x in {1,...,8} {
      \filldraw [black] (\x, 0) circle (2pt);
      \filldraw [black] (\x, 1) circle (2pt);
    }
    \draw (1,1)--(2,0) (2,1)--(3,0);
    \draw (8,1)--(7,0) (7,1)--(6,0);
    \draw (4,1)--(5,0) (5,1)--(4,0);
\end{tikzpicture}
\end{minipage}$
has the associated triple
\[
(\dom(d),\pi(d),\im(d)) = \big(\{1,2,4,5,7,8\}, \pi,
\{2,3,4,5,6,7\}\big)
\]
where $\pi(d) = \pi = \big(\begin{smallmatrix}
  1&2&3&4&5&6\\
  1&2&4&3&5&6
\end{smallmatrix}\big) \in S_6$ in the usual two-line notation.

Given any $(X,\pi,Y)$ such that $X,Y$ are subsets of $\{1, \dots, r\}$
of the same cardinality $k$, and $\pi \in S_k$, there is a
corresponding partial permutation $d$ in $\PP_r$ such that $(\dom(d),
\pi(d), \im(d)) = (X,\pi,Y)$. In other words, $d$ is determined
uniquely by its triple $(\dom(d), \pi(d), \im(d))$.

\begin{lem}\label{lem:mult-props}
  Suppose that $d_1, d_2 \in \PP_r$, regarded as partial
  permutations. Put $Z = \im(d_1) \cap \dom(d_2)$. Then:
  \begin{enumerate}
  \item $\rank(d_1 \circ d_2) = |Z|$.
  \item $\dom(d_1 \circ d_2) = (Z)d_1^{-1}$, the preimage of $Z$ under $d_1$.
  \item $\im(d_1 \circ d_2) = (Z)d_2$, the image of $Z$ under $d_2$.
  \item If $d_1d_2 = z^N(d_1 \circ d_2)$ is the product in $\PP_r(z)$ then
    \[N = r - |\im(d_1) \cup \dom(d_2)|.\]
  \end{enumerate}
\end{lem}

\begin{proof}
Parts (a)--(c) follow by observing that the paths from top to bottom
in the composite diagram formed by stacking $d_1$ above $d_2$ all pass
through the points in $Z = \im(d_1) \cap \dom(d_2)$.  Part (d) follows
by noticing that the connected components in the middle of the
composite diagram is the cardinality of the intersection of the
complements of $\im(d_1)$ and $\dom(d_2)$ in $\{1, \dots, r\}$; then
apply DeMorgan's law.
\end{proof}

For the remainder of this section, we mostly follow the notational
conventions of \cite{Green-Paget}; but we will use $*$ to denote
anti-involutions. We need to work with based vector spaces. For our
purposes, a based vector space is a pair $(U,\sU)$ consisting of a
given vector space $U$ along with a chosen basis $\sU$ (we use script
letters for the basis). Fix the parameters $r$ and $z \ne 0$ and
consider the following data:
\[
\begin{array}{ll}
  A = \PP_r(z)\qquad & \sA = \PP_r \quad \\
  B(k) = \C S_k &   \sB(k) = S_k \\
  U(k) = \C\sU(k) & \sU(k) = \{M \subseteq \{1, \dots, r\}: |M| = k\}. 
\end{array}
\]
Then $(A,\sA)$, $(B(k),\sB(k))$, $(U(k),\sU(k))$ are all based vector
spaces.  The algebra $A$ has a $\C$-linear anti-involution $d \mapsto
d^*$ defined on the diagram basis by letting $d^*$ (for $d \in \sA$)
be the diagram obtained from $d$ by flipping it upside down
(reflecting across its horizontal axis of symmetry). There is a
similar anti-involution (also denoted by $*$) on the group algebra
$B(k) = \C S_k$ defined by $w \mapsto w^{-1}$ for all $w \in \sB(k) =
S_k$. It is clear that
\begin{equation}\label{eq:it-inf-dec}
  A \cong \bigoplus_{k = 0}^r U(k) \otimes B(k) \otimes U(k)
\end{equation}
as vector spaces. Indeed, the isomorphism is defined by sending any $a
\in \sA$ to $\dom(a) \otimes \pi(a) \otimes \im(a)$ and extending linearly.
Thus we may identify $A$ with the right hand side of
\eqref{eq:it-inf-dec}.  It is also clear that (under this
identification) the anti-involutions satisfy the compatibility
condition
\begin{equation}\label{eq:*-compat}
  (u \otimes b \otimes v)^* = v \otimes b^* \otimes u
\end{equation}
for any $u,v \in \sU(k)$, $b \in \sB(k)$, and any $k = 0, 1, \dots, r$.
Now define maps
\[
\phi_k : \sA \times \sU(k) \to U(k), \qquad
\theta_k : \sA \times \sU(k) \to B(k)
\]
by the rules
\begin{align*}
\phi_k(a,u) &= 
\begin{cases}
  z^{r-\rank(a)} (u)a^{-1} & \text{if } u \subset \im(a) \\
  0 & \text{otherwise}
\end{cases} \\
\theta_k(a,u) &= 
\begin{cases}
  \pi(a') & \text{if } u \subset \im(a) \\
  0 & \text{otherwise}
\end{cases}
\end{align*}
where $a' =$ the restriction of $a$ to the preimage $(u)a^{-1}$,
regarded as a partial permutation $(u)a^{-1} \to u$. We claim that the
multiplication rule \eqref{eq:ptn-alg-mult} in the algebra $A =
\PP_r(z)$ satisfies the relation
\begin{equation}\label{eq:mult-rule}
a (u \otimes b \otimes v) \equiv \phi_k(a,u) \otimes \theta_k(a,u)b
\otimes v \mod{J(<k)}
\end{equation}
for any $u,v \in \sU(k)$, $b \in \sB(k)$, $a \in \sA$, where
\[
J(<k) = \bigoplus_{j<k} U(j) \otimes B(j) \otimes U(j).
\]

Recall that $\rank(a)$ is the number of edges in $a$, where $a \in
\sA$.  To see the claim, observe that by Lemma \ref{lem:mult-props}
the rank of $x = a (u \otimes b \otimes v)$ is $|\im(a) \cap u|$,
which is always $\le k$ since $u \in \sU(k)$. Furthermore,
$\rank(x)=k$ if and only if $u \subset \im(a)$, in which case $\pi(x)
= \pi(a') b$ and $x$ has a scalar factor of $z^{r-\rank(a)}$.

Properties \eqref{eq:*-compat}, \eqref{eq:mult-rule} verify the
hypotheses of \cite{Green-Paget}*{Theorem~1}, so by that result $A$ is
a cellular algebra with respect to the anti-involution $*$ and the
cell datum $(\Lambda,M,C)$, where
\[
\Lambda = \{ \lambda: \lambda \vdash k \text{ and } 0 \le k \le r \}
\]
and (for $\lambda \vdash k$)
\[
M(\lambda) = \sU(k) \times \text{Tab}(\lambda),  \qquad
C^\lambda_{(x,X),\,(y,Y)} = x \otimes C^\lambda_{X,Y} \otimes y .
\]
Here, $\text{Tab}(\lambda)$ is the set of standard tableaux of shape
$\lambda$ and $\{C^\lambda_{X,Y}\}$ is any cellular basis of $B(k) = \C
S_k$. (Cellular bases of symmetric group algebras exist; see e.g.,
\cites{GL:96,KL:79,Murphy}.)

\begin{rmk}
Strictly speaking, we should write $\Lambda$ as the set of all pairs
$(k, \lambda)$ such that $k = 0,1, \dots, r$ and $\lambda \vdash
k$. But one can recover the value of $k$ from $\lambda$ itself, so we
have simplified the notation accordingly.
\end{rmk}

The importance of realizing $A = \PP_r(z)$ as an iterated inflation is
that we now have easy access to its representations. This enables an
easy proof of the second part of the following.

\begin{thm}\label{thm:iter-inf}
The algebra $A = \PP_r(z)$ is an iterated inflation of group algebras
of symmetric groups. Furthermore, if $z \ne 0$ we have that:
  \begin{enumerate}
  \item The cell modules of $A$ are indexed by the set $\Lambda$. If
    $\lambda \vdash k$, the cell module $C(\lambda)$ corresponding to
    a given $\lambda \in \Lambda$ has the form
    \[
    C(\lambda) = U(k) \otimes \Specht^\lambda
    \]
    with action for any $a \in \sA$ given by
    \[
    a(x \otimes z) = \phi_k(a,x) \otimes \theta_k(a,x) z
    \]
    for all $x \in \sU(k)$, $z \in \Specht^\lambda$. Here
    $\Specht^\lambda$ is the cell module for $B(k) = \C S_k$ indexed by
    $\lambda$, which may be identified with the (irreducible) Specht
    module indexed by $\lambda$.

  \item Define a $B(k)$-valued $\C$-linear bilinear form $\psi_k$ on
    $U(k)$ by the linear extension of the rule
    \[
    \psi_k(y,u) =
    \begin{cases}
      z^{r-k} \,1_{B(k)} &\text{ if } y=u \\
      0 &\text{ otherwise}
    \end{cases}
    \]
    for any $y,u \in \sU(k)$. Then $\psi_k(y,u)^* = \psi_k(u,y)$ and
    \[
    (x \otimes c \otimes y)(u\otimes b \otimes v) \equiv x \otimes c
    \psi(y,u) b \otimes v \mod{J(<k)}.
    \]

  \item $A$ is semisimple, and the set of cell modules $\{
     C(\lambda): \lambda \in \Lambda \}$ forms a complete set of
     irreducible $A$-modules, up to isomorphism.
  \end{enumerate}
\end{thm}

\begin{proof}
Part (a) is the first half of \cite{Green-Paget}*{Proposition 3}. Part
(b) is essentially the same as \cite{Green-Paget}*{Proposition
  2}. 

For part (c) we use the standard bilinear form $\bil{\cdot}{\cdot}$ on
$C(\lambda)$, which by the second part of
\cite{Green-Paget}*{Proposition 3}, is related to the standard bilinear
form $\bil{\cdot}{\cdot}_\lambda$ on $\Specht^\lambda$ by
\[
\bil{x \otimes z}{y \otimes w} = \bil{z}{\psi_k(x,y)w}_\lambda =
\bil{\psi_k(y,x)z}{w}_\lambda
\]
for any $x,y \in U(k)$, any $z,w \in \Specht^\lambda$. Graham and
Lehrer \cite{GL:96} showed that $A$ is semisimple (and the cell
modules are its irreps) if and only if the bilinear form is
nondegenerate on each cell module. That this is so in our situation
follows from the nondegeneracy of $\bil{\cdot}{\cdot}_\lambda$ and
$\psi_k$. This completes the proof.
\end{proof}

\begin{rmk}
(i) There are other approaches to proving the semisimplicity of
  $\PP_r(z)$. For instance, one could adapt the arguments in
  \cite{Solomon} to find explicit formulas for primitive central
  idempotents in $A = \PP_r(z)$. Another approach is based on
  Corollary \ref{cor:present}, which immediately implies the
  semisimplicity of $\PP_r(z)$ for $z \ne 1$, since Munn proved that
  $\C \PP_r = \PP_r(1)$ is semisimple. We have included the above
  arguments to demonstrate the relevance of the language of cellular
  algebras in providing perspective on existing results in the
  literature.

(ii) The vector space $U(k)$ is itself a $\PP_r(z)$-module, with action
  defined on generators by $a \cdot u = \phi_k(a,u)$ for $a \in \sA$,
  $u \in \sU(k)$.

(iii) Grood \cite{Grood} gives a purely combinatorial construction of
  the irreducible representations of $\C \PP_r = \PP_r(1)$, in terms
  of a natural generalization of standard tableaux. This description
  can easily be adapted to $\PP_r(z)$.

(iv) Observe that the set $\Lambda$ in Theorem \ref{thm:iter-inf}(a)
  is equal to the set $\Lambda_{\le r}$ in Remark \ref{rmk:ptn}(ii).
  In other words, the (isomorphism classes of) irreducible
  representations of $\Ptn_r(z)$ and $\PP_r(z)$ are indexed by the
  same set.
\end{rmk}

To conclude this section, we describe the module structure of $U(k)$.

\begin{prop}\label{prop:U(k)}
$U(k)$ is the vector space spanned by $\sU(k)$, the set of subsets of
  $\{1, \dots, r\}$ of cardinality $k$.  The action of $\PP_r(z)$ on
  $U(k)$ is determined by
\[
w \cdot u = u', \qquad
p_j \cdot u =
\begin{cases}
  z u & \text{ if } j \notin u\\
  0 & \text{ otherwise}
\end{cases}
\]
for all $w \in S_r$, $j = 1, \dots, r$, $u \in \sU(k)$, where
$u' = (u)w^{-1}$ is the preimage of $u$ under $w$.
\end{prop}

\begin{proof}
If $w \in S_r$ then $\im(w) = \{1, \dots, r\}$. Hence $\phi_k(w,u) =
u'$. Furthermore, if $j = 1, \dots, r$ then $\im(p_j) = \{1, \dots,
r\} - \{j\}$. If $u \in \sU_k$ then $u \subset \im(p_j)$ if and only
if $j \notin u$, in which case the preimage of $u$ under $p_j$ is $u$
itself, and $\phi_k(p_j,u) = z u$. Otherwise, $\phi_k(p_j,u)=0$.
\end{proof}

\section{The main result}\label{sec:main}\noindent
In this section we prove our main result, that Schur--Weyl duality
holds for the tensor power $\bE^{\otimes r}$ of the unreduced Burau
representation, assuming that $q_1q_2 \ne 0$ and $q=-q_2/q_1$ is not a
root of unity.

The space $\bE^{\otimes r}$ is a representation of the group $B_n$
with $B_n$ acting diagonally, via:
\begin{equation}\label{eq:B_n-action}
b \cdot (v_1 \otimes \cdots \otimes v_r) = (b \cdot v_1) \otimes
\cdots \otimes (b\cdot v_r)
\end{equation}
for any $b \in B_n$, $v_j \in \bE$ for $j = 1, \dots, r$. It is also a
representation of $\PP_r([n]_q)$, with (left) action defined on
generators by
\begin{equation}\label{eq:part_perm-action}
\begin{aligned}
s_i \cdot (v_1 \otimes \cdots \otimes v_r) &= v_1 \otimes \cdots
\otimes v_{i-1} \otimes v_{i+1} \otimes v_{i} \otimes v_{i+2} \otimes
\cdots \otimes v_r\\
p_j \cdot (v_1 \otimes \cdots \otimes v_r) &= v_1 \otimes \cdots
v_{j-1} \otimes P v_j \otimes v_{j+1} \otimes \cdots \otimes v_r
\end{aligned}
\end{equation}
for any $i = 1, \dots, r-1$, $j = 1, \dots, r$. In the above, $P$ is
the pseudo-projection defined in Theorem \ref{thm:P}, acting solely in
the $j$th tensor place, and the action of the $s_i$ induces the usual
place-permutation action of the symmetric group $S_r$. That
\eqref{eq:part_perm-action} gives a well-defined action of
$\PP_r([n]_q)$ requires proof, as follows.

\begin{lem}\label{lem:pp-action-ok}
  Equations \eqref{eq:part_perm-action} make $\bE^{\otimes r}$ into a
  left $\PP_r([n]_q)$-module.
\end{lem}

\begin{proof}
We only need to check that the operators corresponding to the action
of the generators $s_j$, $p_i$ satisfy the defining relation of
$\PP_r([n]_q)$ in Theorem \ref{thm:present}. The fact that $P^2 =
[n]_q P$ immediately implies that $p_j^2 = [n]_q p_j$. It is clear
from the definition of the action that $p_ip_j = p_jp_i$ for any $i
\ne j$. So the relations in Theorem \ref{thm:present}(a) are
satisfied. The relations in Theorem \ref{thm:present}(b) are standard
for the place-permutation action of symmetric groups. Finally, the
relations in Theorem \ref{thm:present}(c) are easily verified from the
definitions of the action.
\end{proof}

\begin{rmk}
One can also define the action of $\PP_r([n]_q)$ directly, without
invoking Theorem \ref{thm:present}. Given a partial permutation $d \in
\PP_r$, it is clear that
\[
d \cdot (v_1 \otimes \cdots \otimes v_r) = u_1 \otimes
\cdots \otimes u_r
\]
where $u_i = v_j$ if $d$ takes $i$ to $j$ and $u_i = P v_i$
otherwise. See \cite{Solomon}*{Lemma 5.4} for further details.
\end{rmk}

By Theorem \ref{thm:P}, the actions of $P$ and $B_n$ on $\bE$
commute. It follows that the actions defined in \eqref{eq:B_n-action},
\eqref{eq:part_perm-action} commute, so they make $\bE^{\otimes r}$
into a $(\C B_n, \PP_r([n]_q)$-bimodule.

\begin{prop}\label{prop:tens-decomp}
Set $n = \dim \bE$. Assume that $q_1q_2 \ne 0$ and $q = -q_2/q_1$ is
not a root of unity.  As a $(\C B_n, \PP_r([n]_q))$-bimodule, we have
a multiplicity-free semisimple decomposition
\[
\bE^{\otimes r} \cong \bigoplus_{k=0}^r \; \bigoplus_{\lambda \vdash
  k: \, \ell(\lambda) \le n-1} \Delta(\lambda) \otimes C(\lambda) .
\]
\end{prop}

\begin{proof}
We begin by decomposing $\bE^{\otimes r}$ as a $\C B_n$-module. To
that end, we use the projection $P_0$ onto $\bL$ defined in the proof
of Theorem \ref{thm:P}. We have $1 = P_0 + (1-P_0)$ where $1 =
\id_{\bE}$ is the identity map on $\bE$. This orthogonal idempotent
decomposition induces the vector space decomposition
\[
\bE = P_0 \bE \oplus (1-P_0)\bE = \bL \oplus \bF;
\]
see Remark \ref{rmk:P_0}. Now we expand the $r$th tensor power of
the identity map $1 = \id_\bE$ to obtain
\[
1^{\otimes r} = (P_0+(1-P_0))^{\otimes r} = \sum_{X \subset \{1,
  \dots, r\}} \delta_1(X) \otimes \cdots \otimes \delta_r(X)
\]
where we define the symbol $\delta_j(X)$ for $X \subset \{1, \dots,
r\}$ by
\[
\delta_j(X) = 
\begin{cases}
  1-P_0 & \text{if } j \in X\\
  P_0 & \text{otherwise.}
\end{cases}
\]
Applying the maps in the above decomposition to tensor space
$\bE^{\otimes r}$ gives the decomposition
\[
\bE^{\otimes r} = \bigoplus_{X \subset \{1, \dots, r\}} \delta_1(X)\bE
\otimes \cdots \otimes \delta_r(X)\bE.
\]
This is a decomposition as $\C B_n$-modules. Its isotypic components
are of the form
\[
I(k) := \bigoplus_{X \subset \{1, \dots, r\}:\, |X| = k}
\delta_1(X)\bE \otimes \cdots \otimes \delta_r(X)\bE \cong
\binom{r}{k} \bF^{\otimes k} \otimes \bL^{\otimes r-k}
\]
where $k = 0, 1, \dots, r$ and where the isomorphism follows from the
standard commutativity (up to isomorphism) of tensor products.

Each summand of $I(k)$ is a tensor product of $k$ copies of $\bF$ and
$r-k$ copies of $\bL$ in some order, determined by the indexing subset
$X$ of cardinality $k$. For instance, the summand indexed by $X =
\{1,\dots, k\}$ is $\bF^{\otimes k} \otimes \bL^{\otimes
  r-k}$. Plugging in the result of Theorem \ref{thm:SWD-R^k} for
$\bF^{\otimes k}$ we obtain
\[
\bF^{\otimes k} \otimes \bL^{\otimes r-k} \cong \bigoplus_{\lambda
  \vdash k: \, \ell(\lambda) \le n-1} \Delta(\lambda) \otimes
\Specht^\lambda \otimes \bL^{\otimes r-k} .
\]
Every other summand of $I(k)$ is isomorphic to the above, as $\C
B_n$-modules, by an isomorphism that permutes the tensor places
according to the given indexing subset $X$. Thus we have
\[
I(k) = \bigoplus_{X \subset \{1, \dots, r\}:\, |X|=k}
\left( \bigoplus_{\lambda \vdash k: \, \ell(\lambda) \le n-1} \Delta(\lambda)
\otimes \Specht^\lambda \otimes \bL^{\otimes r-k} \right)
\]
where the copy of $\Delta(\lambda) \otimes \Specht(\lambda)$ in each
term on the right hand side is spread across the tensor places in $X$.
Switching the order of summation using commutativity of tensor
products, we obtain that up to isomorphism
\[
I(k) \cong \bigoplus_{\lambda \vdash k: \, \ell(\lambda) \le n-1}
\Delta(\lambda) \otimes \Specht^\lambda \otimes \left(\bigoplus_{X
  \subset \{1, \dots, r\}:\, |X|=k} \bL^{\otimes r-k} \right).
\]
In this decomposition, the copies of $\Delta(\lambda) \otimes
\Specht^\lambda$ in the right hand side are spread across tensor
positions indexed by $X$, while the $r-k$ factors of $\bL$ are in
tensor positions indexed by the complement $\{1, \dots, r\} -
X$. Hence the fact that $PP_0 = [n]_q P_0$ implies that $p_j$ acts on
$\bigoplus_{X \subset \{1, \dots, r\}:\, |X|=k} \bL^{\otimes r-k}$ as
the scalar $[n]_q$ if $j \notin X$ and as zero otherwise. By
Proposition \ref{prop:U(k)} it follows that if we set $z = [n]_q$ and
regard $U(k)$ as a $\PP_r([n]_q)$-module, then
\[
U(k) \cong \bigoplus_{X \subset \{1, \dots, r\}:\, |X|=k} \bL^{\otimes
  r-k}.
\]
The result thus follows from Theorem \ref{thm:iter-inf}. The proof is
complete.
\end{proof}

Now we finally obtain our main result.

\begin{thm}\label{thm:main}
Suppose that $q_1q_2 \ne 0$ and $q = -q_2/q_1$ is not a root of unity.
As a $(\C B_n, \PP_r([n]_q))$-bimodule, $\bE^{\otimes r}$ satisfies
Schur--Weyl duality, in the sense that the enveloping algebra of each
action is equal to the centralizer algebra for the other:
\[
\im(\C B_n) = \End_{\PP_r([n]_q)}(\bE^{\otimes r}), \quad
\im(\PP_r([n]_q)) = \End_{B_n}(\bE^{\otimes r}).
\]
The action of $\PP_r([n]_q)$ is faithful if and only if $n > r$, in
which case we have an isomorphism $\PP_r([n]_q)
\cong \End_{B_n}(\bE^{\otimes r})$.
\end{thm}

\begin{proof}
Since the actions of $\C B_n$, $\PP_r([n]_q)$ commute, we have
inclusions
\[
\im(\C B_n) \subset \End_{\PP_r([n]_q)}(\bE^{\otimes r}), \quad
\im(\PP_r([n]_q)) \subset \End_{B_n}(\bE^{\otimes r}).
\]
By the standard theory of (split) semisimple algebras, it suffices to
show that either of these inclusions is equality. By Proposition
\ref{prop:tens-decomp} and Schur's Lemma, we have
\begin{equation}\label{eq:dim-End}
\dim \End_{B_n}(\bE^{\otimes r}) = \sum_{k=0}^r \sum_{\lambda \vdash
  k:\, \ell(\lambda) \le n-1} (\dim C(\lambda))^2 .
\end{equation}
If $n>r$ then the length restriction in the second sum above is
vacuous, and the right hand side is equal to the dimension of
$\PP_r([n]_q)$, so the equality
\[
\im(\PP_r([n]_q)) = \End_{B_n}(\bE^{\otimes r})
\]
follows by dimension comparison, and furthermore the action of
$\PP_r([n]_q)$ is faithful. This completes the proof if $n>r$. 

Otherwise the action of $\PP_r([n]_q)$ is not faithful, and by
Proposition \ref{prop:tens-decomp} the $C(\lambda)$ for $\ell(\lambda)
> n-1$ do not appear as constituents of the representation
$\bE^{\otimes r}$. Hence their corresponding matrix algebras appear in
the kernel of the representation $\PP_r([n]_q) \to \End(\bE^{\otimes
  r})$. Thus the right hand side of equation \eqref{eq:dim-End} is
still equal to the dimension of the image of the representation, so
again the needed equality follows by dimension comparison.
\end{proof}

In particular, the algebra $\End_{B_n}(\bE^{\otimes r})$ is spanned by
endomorphisms arising from the (appropriately scaled) action of rook
monoid elements, and is generated by the endomorphisms coming from the
place-permutation action of $S_r$ and by $p_1$, since the other $p_j$
are all conjugate to $p_1$ by elements of $S_r$.

\begin{rmk}\label{rmk:scaling}
(i) Bowing to tradition in the literature on diagram algebras, we have
formulated our main results in terms of the pseudo-idempotent operator
$P$ of Theorem \ref{thm:P}. From some viewpoints (clearing
denominators, compatibility with the standard definitions of the
partition and Brauer algebras) it makes sense to use $P$ instead of
the actual projection operator $P_0$. On the other hand, there is
nothing to prevent from using $P_0$ in place of $P$. Doing so leads to
a statement of Schur-Weyl duality in which the semigroup algebra
$\PP_r(1)$ plays the role of $\PP_r([n]_q)$. We leave the details to
the reader.

(ii) Solomon \cite{Solomon} proved an instance of Schur--Weyl duality
for the rook monoid. In terms of our notation, he considers
$\bE^{\otimes r}$ as a representation of $\GL(\bF)$ with $\GL(\bF)$
acting trivially on $\bL$, and computes the centralizer
$\End_{\GL(\bF)}(\bE^{\otimes r})$, proving that it is a quotient of
the semigroup algebra of the rook monoid. Thus his Schur--Weyl duality
is closely related to ours, although our approach is self-contained
and independent of \cite{Solomon}.
\end{rmk}

We conclude this section by considering the effects of setting $q=1$.
In that case, the representation $B_n \to \End(\bE^{\otimes r})$
factors through the Weyl group $W_n \subset \GL(\bE)$, where $W_n$ is
isomorphic to the symmetric group on $n$ letters, and the Burau
representation $\bE$ becomes isomorphic to the standard
$n$-dimensional representation of $W_n$. Here, we know by the work of
Martin and Jones (see \cite{HR:05}) that $\End_{B_n}(\bE^{\otimes r})
= \End_{W_n}(\bE^{\otimes r})$ is the image of the full partition
algebra $\Ptn_r(n)$ at parameter $n$.

On the other hand, at $q=1$ we have $[n]_q = n$, and one can ask for a
determination of the centralizer algebra $\End_{\PP_r(n)}(\bE^{\otimes
  r})$ for $\PP_r(n)$. At first glance this appears to be a separate
problem, but in fact it leads to yet another new instance of
Schur--Weyl duality, as follows.

\begin{cor}\label{cor:SWD:q=1}
For any $n>2$, there exists a complex number $q$ such that $[n]_q = n$
and $q$ is not a root of unity. Fix such a value of $q = -q_2/q_1$ in
the Burau representation.  Regarded as a $(\C B_n,
\PP_r(n))$-bimodule, $\bE^{\otimes r}$ satisfies Schur--Weyl duality,
in the sense that the enveloping algebra of each action is equal to
the centralizer algebra for the other:
\[
\im(\C B_n) = \End_{\PP_r(n)}(\bE^{\otimes r}), \quad
\im(\PP_r(n)) = \End_{B_n}(\bE^{\otimes r}).
\]
The action of $\PP_r(n)$ is faithful if and only if $n > r$, in which
case we have an isomorphism $\PP_r(n) \cong \End_{B_n}(\bE^{\otimes
  r})$.
\end{cor}

\begin{proof}
Values of $q$ such that $[n]_q = n$ are obtained by solving the
polynomial equation
\[
1+x+x^2+\cdots+x^{n-1} = n \iff (1-n)+x+x^2+\cdots+x^{n-1} = 0.
\]
If $n>2$ then non-root of unity solutions exist, as the product of the
roots of the polynomial is, up to sign, equal to $|1-n|=n-1$ and
$n-1>1$. Let $q$ be such a solution. Then $[n]_q = n$ and hence
$\PP_r([n]_q) = \PP_r(n)$. The result then follows from Theorem
\ref{thm:main}. 
\end{proof}

\begin{rmk}
Suppose that $n > 2$, and fix a non root of unity value of $q$, as in
Theorem \ref{cor:SWD:q=1}. Then the algebra
$\End_{\PP_r(n)}(\bE^{\otimes r})$ does not depend on $q$, but is
spanned by linear endomorphisms of the form $g \otimes \cdots \otimes
g$, where $g$ is the image of a braid group element under the Burau
representation $B_n \to \GL(\bE)$, which depends on $q$.
\end{rmk}

\section{Dimensions of $\PP_r(z)$-irreps}\label{sec:dims}\noindent
Assume that $q$ is not a root of unity and $q_1q_2 \ne 0$. For any
$z\ne 0$, Theorem \ref{thm:iter-inf}(a) gives one way of computing
dimensions of the irreducible $\PP_r(z)$-modules, in terms of
dimensions of Specht modules for symmetric groups.  We now apply the
Schur--Weyl duality statement in Theorem \ref{thm:main} to obtain a
second, more combinatorial, method of computing those dimensions.  By
Schur--Weyl duality, we know that:
\begin{quote}
\emph{The multiplicities of the irreducible representations of
$B_n$ in $\bE^{\otimes r}$ are equal to the dimensions of the
irreducible representations of the corresponding centralizer
algebra.}
\end{quote}
The proof of Proposition \ref{prop:tens-decomp} shows that this set of
modules is precisely the set of irreducible polynomial $\GL(\bF) \cong
\GL_{n-1}(\C)$-modules appearing in $\bigoplus_{k=0}^r \bF^{\otimes
  k}$. Note that the action of $B_n$ on this module is given by
composition with the representation $B_n \to \GL(\bF)$.  Thus the set
\[
\Lambda(n,r) = \{\lambda \vdash k: 0 \le k \le r \text{ and }
\ell(\lambda) \le n-1\}
\]
indexes the set of modules in question. By semisimplicity, we may
write
\[
  \bE^{\otimes r} \cong \bigoplus_{\lambda \in \Lambda(n,r)}
  c^{n,r}_\lambda \Delta(\lambda)
\]
where $c^{n,r}_\lambda = \dim \text{Hom}_{B_n}(\Delta(\lambda),
\bE^{\otimes r}) = \dim \text{Hom}_{\GL(\bF)}(\Delta(\lambda),
\bE^{\otimes r})$ is the desired multiplicity.

We may assume that $n>r$, in which case the value of $c^{n,r}_\lambda$
does not depend on $n$, so we write $c^r_\lambda$ for this limiting
value. These numbers may be computed by induction on $r$ (holding
$n>r$ fixed) using the well-known Pieri rule \cite{FH} for tensor
products of polynomial $\GL_{n-1}(\C)$ representations. Notice that
for $n>r$ the indexing set $\Lambda(n,r)$ is equal to
\[
 \Lambda(r) = \{\lambda \vdash k: 0 \le k \le r\},
\]
which also does not depend on $n$. This is the indexing set of the
irreducible $\PP_r(z)$-modules (up to isomorphism). Set
$\delta_{r,\lambda} = 1$ if $\lambda \in \Lambda(r)$ and $0$
otherwise. Then for any $\lambda \in \Lambda(r)$, we claim that
\begin{equation}\label{eq:mult-rec}
  c^r_\lambda = \delta_{r-1,\lambda}\, c^{r-1}_\lambda + \textstyle
  \sum_\mu c^{r-1}_\mu
\end{equation}
where the summation is over all $\mu \in \Lambda(r-1)$ obtained from
$\lambda$ by removing one box from its Young diagram. This recursion
computes $\dim C(\lambda) = c^r_\lambda$, where $C(\lambda)$ is the
irreducible $\PP_r([n]_q)$-module indexed by $\lambda$. Since for $z
\ne 0$ we have $\PP_r([n]_q) \cong \PP_r(z)$ by Corollary
\ref{cor:present}, this also computes the dimension of the
corresponding irreducible $\PP_r(z)$-module.

To prove the claim, we may as well specialize $(q_1,q_2)$ to
$(1,-q)$. Then $\bL \cong \C$ is isomorphic to the trivial
module. Assume by induction that $\bE^{\otimes r-1} \cong
\bigoplus_{\mu \in \Lambda(r-1)} c^{r-1}_\lambda
\Delta(\mu)$. Tensoring both sides by $\bE = \bL \oplus \bF \cong \C
\oplus \bF$ gives
\[
\bE^{\otimes r} \cong \bigoplus_{\mu \in \Lambda(r-1)}
c^{r-1}_\mu \Delta(\mu) \oplus \bigoplus_{\mu \in
  \Lambda(r-1)} c^{r-1}_\mu \Delta(\mu)\otimes \bF.
\]
The Pieri rule (in a special case) says that $\Delta(\mu) \otimes \bF$
is isomorphic to a direct sum of $\Delta(\lambda)$, each with
multiplicity one, for every partition $\lambda$ obtained from $\mu$ by
adding one box to its Young diagram. Combining this with the above
decomposition justifies equation \eqref{eq:mult-rec}. Tabulating the
values of $c^r_\lambda$ for $r = 0,1,2,3,4$ gives the following table,
in which the number in the final column is the sum of squares of the
$c^r_\lambda$ in its row.
\[\setlength{\arraycolsep}{1pt}
\begin{array}{c|cccccccccccc|c}
  \;r\; &\;\emptyset\;
  &(1)&(2)&(1^2)&(3)&(2,1)&(1^3)&(4)&(3,1)&(2,2)&(2,1^2)&(1^4)\;
  & \;\dim \PP_r(z) \\ \hline
  0 & 1 &&&&&&&&&&&& 1 \\
  1 & 1&1&&&&&&&&&&& 2 \\
  2 & 1&2&1&1&&&&&&&&& 7\\
  3 & 1&3&3&3&1&2&1&&&&&& 34\\
  4 & 1&4&6&6&4&8&4&1&3&2&3&1 & 209
\end{array}
\]
We can also compute these numbers using a Bratteli diagram. For our
purposes, a Bratteli diagram is a rooted tree with vertices given by
Young diagrams, constructed by the following recursive algorithm,
assuming that the first $r-1$ rows are already constructed:
\begin{itemize}
\item Copy the vertices in row $r-1$ to row $r$, and add an edge from
  each copied vertex $\mu$ to its corresponding vertex $\mu$
  in the previous row.
\item Add the partitions of $r$ to the end of the $r$th row, and draw
  an edge from vertex $\lambda$ in the $r$th row to vertex $\mu$ in
  the preceding row if and only if $\lambda$ is obtainable from $\mu$
  by adding one box.
\end{itemize}
To get started, the initial row $r=0$ contains a single vertex labeled
by the empty partition $\emptyset$. We display the first 5 rows of
this graph in Figure~1 below,
\begin{figure}[ht]
\begin{center}
\begin{tikzpicture}[xscale=2.6*\UNIT, yscale=-5*\UNIT]
  \coordinate (0) at (0,0);
  \foreach \x in {0,...,1}{\coordinate (1\x) at (2*\x,2);}
  \foreach \x in {0,...,3}{\coordinate (2\x) at (2*\x,4);}
  \foreach \x in {0,...,6}{\coordinate (3\x) at (2*\x,6);}
  \foreach \x in {0,...,11}{\coordinate (4\x) at (2*\x,8);}
  \draw (0)--(10) (0)--(11);
  
  \draw (10)--(20) (10)--(21) (11)--(21) (11)--(22) (11)--(23);
  
  \draw (20)--(30) (20)--(31) (21)--(31) (21)--(32) (21)--(33);
  \draw (22)--(32) (22)--(34) (22)--(35);
  \draw (23)--(33) (23)--(35) (23)--(36);

  \draw (30)--(40) (30)--(41) (31)--(41) (31)--(42) (31)--(43);
  \draw (32)--(42) (32)--(44) (32)--(45);
  \draw (33)--(43) (33)--(45) (33)--(46);

  \draw (34)--(44) (34)--(47) (34)--(48); 
  \draw (35)--(45) (35)--(48) (35)--(49) (35)--(410);
  \draw (36)--(46) (36)--(410) (36)--(411);
\def\NULL{\footnotesize$\emptyset$}
\begin{scope}[every node/.style={fill=white}]
  \node at (0) {\NULL}; 
  
  \node at (10) {\NULL};
  \node at (11) {\PART{1}}; 

  \node at (20) {\NULL};
  \node at (21) {\PART{1}};
  \node at (22) {\PART{2}};
  \node at (23) {\PART{1,1}};
  
  \node at (30) {\NULL};
  \node at (31) {\PART{1}};
  \node at (32) {\PART{2}};
  \node at (33) {\PART{1,1}};
  \node at (34) {\PART{3}};
  \node at (35) {\PART{2,1}};
  \node at (36) {\PART{1,1,1}};

  \node at (40) {\NULL};
  \node at (41) {\PART{1}};
  \node at (42) {\PART{2}};
  \node at (43) {\PART{1,1}};
  \node at (44) {\PART{3}};
  \node at (45) {\PART{2,1}};
  \node at (46) {\PART{1,1,1}};
  \node at (47) {\PART{4}};
  \node at (48) {\PART{3,1}};
  \node at (49) {\PART{2,2}};
  \node at (410) {\PART{2,1,1}};
  \node at (411) {\PART{1,1,1,1}};
\end{scope}
\end{tikzpicture}
\end{center}
\caption{Bratteli diagram for $r=4$}
\label{Bratteli}
\end{figure}
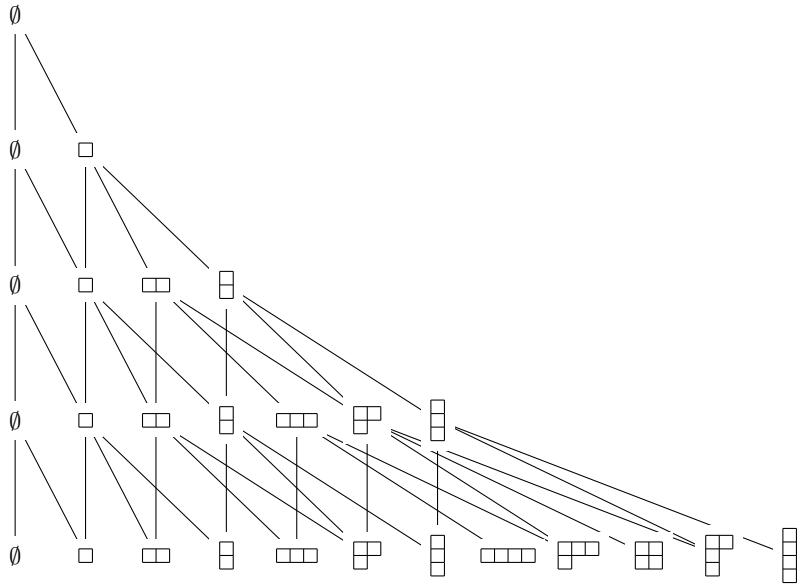
noting that the dimension of the irreducible $\PP_r(z)$-module
$C(\lambda)$ is equal to the number of paths of $r$ edges from the
root to $\lambda$ in the graph. Figure 1 coincides with
\cite{Halverson}*{Fig.~1}.

\begin{rmk}
The above method also computes the $c^{n,r}_\lambda$ if $n \le
r$. Simply observe that $c^{n,r}_\lambda = 0$ if $\ell(\lambda) \ge n$
and is equal to $c^r_\lambda$ otherwise, so to compute dimensions of
the irreducible representations of the centralizer algebra
$\End_{B_n}(\bE^{\otimes r})$, we merely omit all nodes from the
Bratteli diagram with $n$ or more rows of boxes.
\end{rmk}


\section{Analogues of Schur algebras}\noindent
We continue to assume that $q_1q_2 \ne 0$ and $q = -q_2/q_1$ is not a
root of unity. To avoid complexity, it is often useful to replace the
infinite-dimensional group algebra $\C B_n$ in Theorem \ref{thm:main}
by its finite-dimensional image $S'_q(n,r)$. So we let $S'_q(n,r)$ be
the enveloping algebra of the $B_n$-action; that is, the image of the
representation $\C B_n \to \End(\bE^{\otimes r})$.  Then Schur--Weyl
duality implies that
\begin{equation*}
S'_q(n,r) = \End_{\PP_r([n]_q)}(\bE^{\otimes r})
\end{equation*}
as algebras. Regarding $\bE^{\otimes r}$ as an $(S'_q(n,r),
\PP_r([n]_q))$-bimodule, notice that Schur--Weyl duality still holds,
so we lose no information by this replacement. Since $\PP_r([n]_q)$
contains a copy of the symmetric group $S_r$, it is clear that
$S'_q(n,r)$ is a subalgebra of the well-known Schur algebra (see
\cites{Green,Martin})
\[
S(n,r) = \End_{S_r}(\bE^{\otimes r})
\]
appearing in classical Schur--Weyl duality.  Since all the $p_j$
defined in \eqref{eq:part_perm-action} are conjugate in $\PP_r(z)$ for
any $z$, it follows that
\begin{equation}
S'_q(n,r) = S(n,r) \cap \text{Comm}(p_1) = \{X \in S(n,r): p_1X = Xp_1\}
\end{equation}
is the subalgebra of $S(n,r)$ consisting of elements that commute with
the operator $p_1$.

\begin{example}
The Schur algebra $S(2,2)$ is 10-dimensional, with
\cite{Martin}*{\S2.2} a faithful matrix realization in $\End(\bE
\otimes \bE)$ as the set of matrices of the form
\[
\begin{bmatrix}
  x_1 & x_2 & x_2 & x_3\\
  x_4 & x_5 & x_6 & x_7\\
  x_4 & x_6 & x_5 & x_7\\
  x_8 & x_9 & x_9 & x_{10}
\end{bmatrix}
\]
where $x_1, \dots, x_{10}$ are arbitrary. Taking the intersection with
$\text{Comm}(p_1)$ shows that $S'_q(2,2)$ is the set of matrices of
the form
\[
\begin{bmatrix}
  x_1 & qx_4 & qx_4 & q^2 x_8\\
  x_4 & x_1 + (q-1)x_4 & qx_8 & qx_4 + (q-1)qx_8\\
  x_4 & qx_8 & x_1 + (q-1)x_4 & qx_4 + (q-1)qx_8\\
  x_8 & x_4 + (q-1)x_8 & x_4 + (q-1)x_8 & x_1 + 2(q-1)x_4 + (q-1)^2x_8
\end{bmatrix}
\]
where $x_1, x_4, x_8$ are arbitrary. So $\dim S'_q(2,2) = 3$. 
\end{example}

The irreducible $S'_q(n,r)$-modules are precisely the restrictions to
$\C B_n$ of the homogeneous polynomial representations of $\GL(\bF)
\cong \GL_{n-1}(\C)$ of degrees $0, 1, \dots, r$ and are indexed by
the set $\Lambda(n,r)$.

Suppose that $n>2$ and assume that $q$ is chosen (as in Corollary
\ref{cor:SWD:q=1}) so that $[n]_q = n$ while maintaining the condition
that $q$ is not a root of unity. In this situation, part of the
Schur--Weyl duality statement in Corollary \ref{cor:SWD:q=1} is that
the enveloping algebra $S'(n,r)$ of the $B_n$-action is the
centralizer algebra
\[
S'(n,r) = \End_{\PP_r(n)}(\bE^{\otimes r}). 
\]
This algebra is the specialization at $q=1$ of the algebra
$S'_q(n,r)$.  Notice that it is independent of $q$ because the action
of $\PP_r(n)$ is independent of $q$. This subalgebra of the classical
Schur algebra makes sense over any commutative ring. It may be
interesting to characterize its irreducible representations over a field
of positive characteristic.

\appendix
\section{Proof of Theorem \ref{thm:density}}\label{sec:proof}\noindent 
Assume that $q_1q_2 \ne 0$ and $q = -q_2/q_1$ is not a root of
unity. In this appendix (which is independent of Sections
\ref{sec:Ptn-alg}--\ref{sec:dims}) we give an elementary proof that
the Zariski-closure $\ov{G}$ contains $\SL(\bF)$, where $G =
\rho(B_n)$ is the image of the reduced Burau representation $\rho: B_n
\to \GL(\bF)$. This result is needed to obtain Schur--Weyl duality for
$\bF^{\otimes k}$ (see Theorem \ref{thm:SWD-R^k}), which is needed for
the proof of the main results in Section \ref{sec:main}.  In case $q$
is transcendental, the result is a special case of
\cite{Marin:07}*{Theorem~B}, obtained by different methods.

We get the following result from Lemma \ref{lem:action-on-R} by an
elementary induction on $k$, which is left to the reader.

\begin{lem}\label{lem:powers-of-sigma}
  Let $k$ be a positive integer. For any $i,j = 1, \dots, n-1$ the
  action of $\sigma_i^k \in B_n$ on $\ff_j \in \bF$ is given by the
  rules
  \begin{enumerate}
  \item  $\rho(\sigma_i^k) \ff_j = q_1^k \ff_j$  if $j \ne
    i-1, i, i+1$.
  \item  $\rho(\sigma_i^k) \ff_{i-1} = q_1^k \ff_{i-1} + q_1
    \Phi_k(q_1,q_2) \ff_i$.
  \item  $\rho(\sigma_i^k) \ff_i = q_2^k \ff_i$.
  \item  $\rho(\sigma_i^k) \ff_{i+1} = -q_2 \Phi_k(q_1,q_2) \ff_i
    + q_1^k\ff_{i+1}$
  \end{enumerate}
  where $\Phi_k(q_1,q_2) = \sum_{j=0}^{k-1} q_1^j q_2^{k-1-j} =
  \frac{q_1^k-q_2^k}{q_1-q_2}$.  
\end{lem}

Set $\Phi_k = \Phi_k(q_1,q_2)$ for short. For convenience of
reference, the matrices of the operators $\rho(\sigma_i^k)$ with
respect to the $\{\ff_i\}$ basis are listed below:
\begin{align*}
\rho(\sigma_1^k) &=
\begin{bmatrix}
  q_2^k & -q_2\Phi_k &  0  & \cdots & 0\\
  0   &  q_1^k &  0  & \cdots & 0\\
  0   &  0   & q_1^k & \cdots & 0\\
  \vdots&\vdots&\vdots&\ddots&\vdots\\
  0   &  0   & 0 & \cdots & q_1^k
\end{bmatrix} , \ {}
\rho(\sigma_{n-1}^k) =
\begin{bmatrix}
  q_1^k & \cdots & 0 & 0 & 0 \\
  \vdots & \ddots & \vdots & \vdots & \vdots \\
  0 & \cdots & q_1^k & 0 & 0 \\
  0 & \cdots & 0 & q_1^k & 0 \\
  0 & \cdots & 0 & q_1\Phi_k & q_2^k \\
\end{bmatrix}, \\
\rho(\sigma_i^k) &=
\begin{bmatrix}
  q_1^k & \cdots & 0 & 0  &  0  & \cdots & 0\\
  \vdots & \ddots & \vdots & \vdots & \vdots & \ddots & 0\\
  0 & \cdots & q_1^k & 0  &  0  & \cdots & 0\\
  0 & \cdots & q_1\Phi_k & q_2^k &  -q_2\Phi_k  & \cdots & 0\\
  0 & \cdots & 0   &  0   & q_1^k &  \cdots & 0\\
  \vdots & \ddots &\vdots&\vdots&\vdots&\ddots&\vdots\\
  0 & \cdots & 0   &  0   & 0 & \cdots & q_1^k
\end{bmatrix} \qquad (1 < i < n-1) .
\end{align*}
Henceforth we identify $\GL(\bF) \cong \GL_{n-1}(\C)$ by means of the
basis $\{\ff_i\}$.

Recall (Remark \ref{rmk:det}) that $\det(\sigma_j) = q_1^{n-2}q_2$ for
all $j$.  The proof of Theorem \ref{thm:density} falls naturally into
two cases, depending whether $q_1^{n-2}q_2$ is or is not a root of
unity.  The non root of unity case requires the following lemma.

\begin{lem}\label{lem:non-root-of-unity}
  Suppose that $q_1^{n-2}q_2$ is not a root of unity.  If $g\in G$, $0
  \ne z\in \C$ then $zg=zI_{n-1}g\in \ov{G}$.
\end{lem}

\begin{proof}
It is well known \cite{Kassel-Turaev} that the ``full twist''
$\theta_n = \Delta_n^2 \in B_n$, where
\[
\Delta_n = (\sigma_1 \cdots \sigma_{n-1})(\sigma_1 \cdots
\sigma_{n-2}) \cdots (\sigma_1\sigma_2) \sigma_1,
\]
generates the center of the braid group $B_n$.  (An easy exercise
\cite{Kassel-Turaev}*{Exercise 1.3.2} gives the alternate formula
$\theta_n = (\sigma_1 \cdots \sigma_{n-1})^n$.)  Thus by Schur's Lemma
it follows that $\rho(\theta_n)=\lambda I_{n-1}$ for some $\lambda \in
\C$. Now an easy calculation shows that $\lambda =
(q_1^{n-2}q_2)^n$. As $\lambda$ is not a root of unity, the
Zariski-closure of the subgroup generated by $\rho(\theta_n)$ is
$H=\{zI_{n-1}\,:\, z\neq 0\}$ and so $zI_{n-1} \in \overline{G}$ for
every $0 \ne z\in \C$. Since $\ov{G}$ is a group, for any $g\in G$ it
follows that $zg = g zI_{n-1} \in \overline{G}$.
\end{proof}

Let $\{e_{ij}\}_{i,j=1}^{n-1}$ be the standard basis of matrix units
for matrix space, defined in terms of the Kronecker delta symbols by
$e_{ij} = (\delta_{ik}\delta_{jl})_{k,l=1}^{n-1}$.  Set
\[
E(i) = I_{n-1}-e_{ii} = \textstyle \sum_{j \ne i} e_{jj}.
\]
We also need the constants
\[
a = \frac{q}{1+q}, \qquad b = \frac{1}{1+q}.
\]
Notice that $a = -q_2/(q_1-q_2)$, $b = q_1/(q_1-q_2)$. 

\begin{prop}
Assume that $q_1q_2 \ne 0$ and $q = -q_2/q_1$ is not a root of
unity. Let $\ov{G}$ be the Zariski-closure of $G = \rho(B_n)$. 
\begin{enumerate}
\item If $q_1^{n-2}q_2$ is not a root of unity then $\ov{G}$ contains
  the one-parameter subgroups $H_1, \dots, H_{n-1}$ where
  \begin{align*}
  H_i = (1-\delta_{i,1})&b(1-z_i)e_{i,i-1}+z_i e_{ii}
  \\ &+(1-\delta_{i,n-1})a(1-z_i)e_{i,i+1} +E(i)
  \end{align*}
  for nonzero complex scalars $z_1,\ldots z_{n-1}$.
  
\item If $q_1^{n-2}q_2$ is a root of unity then $\ov{G}$ contains
  the one-parameter subgroups $K_1, \dots, K_{n-1}$ where
  \begin{align*}
  K_i = (1-\delta_{i,1})&b(w_i-w_i^{2-n})e_{i,i-1}+w_i^{2-n}
  e_{ii}\\ &+(1-\delta_{i,n-1})a(w_i-w_i^{2-n})e_{i,i+1} + w_iE(i)
  \end{align*}
  for nonzero complex scalars $w_1,\ldots w_{n-1}$.
\end{enumerate}
\end{prop}

\begin{proof}
(a)
It follows from Lemma \ref{lem:non-root-of-unity} that $q_1^{-1}
\rho(\sigma_i)\in \ov{G}$ for all $i=1,\ldots,n-1$. Hence
$(q_1^{-1}\rho(\sigma_i))^k = q_1^{-k} \rho(\sigma_i^k)$ belongs to
$\ov{G}$ for all $i$ and all $k\ge 0$. By Lemma
\ref{lem:powers-of-sigma} it follows by an easy calculation that
\begin{align*}
(q_1^{-1}\rho(\sigma_i))^k &= (1-\delta_{1,i})b(1-(-q)^k)e_{i,i-1} +
  (-q)^ke_{ii}\\ &\quad + (1-\delta_{i,n-1})a(1-(-q)^k)e_{i,i+1} + E(i).
\end{align*}
Notice that this matrix depends only on $q = -q_2/q_1$ and lies in
$H_i$ for any $k\geq 0$. Since $-q$ is not a root of unity, it follows
that the intersection $H_i \cap \overline{G}$ is infinite. Thus $H_i
\subset \overline{G}$, since $H_i$ is a closed one-parameter group.

(b)
Since $q_1^{n-2}q_2$ is a root of unity, there is a positive integer
$d$ such that $(q_1^{n-2}q_2)^d = 1$, hence
$q_2^{d}=q_1^{d(2-n)}$. Since $q=-q_2/q_1$, it follows that
\[ 
q^{d} = (-1)^d \, q_1^{d(1-n)}.
\]
Now $q$ is not a root of unity, so $q_1$ cannot be a root of unity.
By Lemma \ref{lem:powers-of-sigma}, if $k$ is a non-negative integer
we have
\begin{align*}
\rho(\sigma_i^{kd}) &= 
 (1-\delta_{i,1})b(q_1^{kd}-q_1^{kd(2-n)})e_{i,i-1} + q_1^{kd(2-n)}e_{ii} \\
  &\quad + (1-\delta_{i,n-1})a(q_1^{kd}-q_1^{kd(2-n)})e_{i,i+1} + q_1^{kd} E(i).
\end{align*}
Since $q_1$ is not a root of unity, the powers $q_1^{kd}$ are distinct
for all $k$ and thus the matrices $\rho(\sigma_i^{kd})$ are also
distinct for all $k$. Notice that $\rho(\sigma_i^{kd})$ belongs to
$K_i$ for all $k$. Hence the cardinality of $K_i \cap \overline{G}$ is
infinite. This forces $K_i \subset \ov{G}$, as $K_i$ is a closed
one-parameter group.
\end{proof}

\begin{cor}\label{cor:u_i}
  Assume that $q_1q_2 \ne 0$ and $q = -q_2/q_1$ is not a root of
  unity. Let $\ov{G}$ be the Zariski-closure of $G = \rho(B_n)$. 
  \begin{enumerate}
  \item If $q_1^{n-2}q_2$ is not a root of unity 
  then the Lie algebra $\Lie(\ov{G})$
  contains the elements
  \[
  u_i = (1-\delta_{i,1}) b e_{i,i-1} + e_{ii} + (1-\delta_{i,n-1}) a
  e_{i,i+1} 
  \]
  for $i = 1, \dots, n-1$.

  \item If $q_1^{n-2}q_2$ is a root of unity then $\Lie(\ov{G})$
    contains the elements
    \begin{align*}
    v_i = (1-\delta_{i,1})& b(n-1) e_{i,i-1} +
    (2-n)e_{ii} \\& + (1-\delta_{i,n-1}) a(n-1) e_{i,i+1} + E(i)
    \end{align*}
    for $i = 1, \dots, n-1$.
  \end{enumerate}
\end{cor}

\begin{proof}
For (a), take the derivative $d/dz_i$ at $z_i=1$ of the one-parameter
subgroup $H_i$.  Similarly, for (b) take the derivative $d/dw_i$ at
$w_i=1$ of the one-parameter subgroup $K_i$.
\end{proof}

\begin{prop}\label{prop:Lie-alg}
  Assume that $q_1q_2 \ne 0$ and  $q$ is not a root of unity.
  Suppose that $n \ge 3$. 
  \begin{enumerate}
  \item The Lie algebra generated by $u_1, \dots, u_{n-1}$ equals
    $\gl_{n-1}$.
  \item The Lie algebra generated by $v_1, \dots, v_{n-1}$ equals
    $\mathfrak{sl}_{n-1}$.
  \end{enumerate}
\end{prop}

\begin{proof}
(a)
Let $\g$ be the Lie algebra generated by $u_1, \dots, u_{n-1}$. Since
$\g \subseteq \gl_{n-1}$, it suffices to show the reverse
containment. We will argue that $e_{ij} \in \g$, for all $i,j = 1,
\dots, n-1$, making use of the standard commutator formula $[e_{ij},
  e_{kl}] = \delta_{jk} e_{il} - \delta_{li} e_{kj}$.

Assume that $n \ge 4$.  Direct computation shows that
\[
[u_1, [u_1,u_2]] + [u_1,u_2] = 2a(ab - 1) e_{12} - 2a^2 e_{13}. 
\]
As $a \ne 0$, it follows that
\[
A'_2 = \tfrac{1}{2a}\big( [u_1, [u_1,u_2]] + [u_1,u_2] \big) = (ab - 1)
e_{12} - a e_{13}
\]
is in $\g$. Then $A_1 = u_1$ and $A_2 = bu_1 - A'_2 =
be_{12}+e_{13}+ae_{14}$ are both in $\g$. Clearly $A_1, A_2$ are
linearly independent.

Next we recursively compute elements $A_k$ for $k = 2,
\dots, n-1$ by defining
\[
A_k = \tfrac{1}{a} [A_{k-1},u_k] \qquad \text{for all } k = 3, \dots, n-1. 
\]
Then by direct computation we have
\[
A_k =
\begin{cases}
  b e_{1,k-1} + e_{1k} + a e_{1,k+1} & \text{ if } k = 3, \dots, n-2\\
  b e_{1,n-2} + e_{1,n-1} & \text{ if } k = n-1.
\end{cases}
\]
We claim that the elements $A_1, A_2, \dots, A_{n-1}$ are linearly
independent.

To see this, identify each $A_k$ with its coordinate vector with
respect to the basis $\{e_{11}, \dots, e_{1,n-1}\}$ of the first row
of matrix space. Let $M$ be the matrix whose rows are those coordinate
vectors in order. Then $M$ is the $(n-1) \times (n-1)$ tridiagonal
banded matrix
\[
M=
\begin{bmatrix}
  1&a&0&\cdots&0\\
  b&1&a&\ddots&\vdots\\
  0&b&1&\ddots&0\\
  \vdots&\ddots&\ddots&\ddots&a\\
  0&\cdots&0&b&1
\end{bmatrix} 
\]
with $a$'s on the super-diagonal and $b$'s on the
sub-diagonal. Setting $D_n = \det M$ we see by equation
\eqref{eq:tri-rec} that $D_n$ satisfies
\begin{equation}\label{eq:recursion}
D_n = D_{n-1} - ab D_{n-2} \qquad(n\ge 5) 
\end{equation}
where $D_3 = 1-ab$ and $D_4 = 1-2ab$. Thus we obtain the following.

\begin{lem}\label{lem:claim}
$D_n = \dfrac{[n]_q}{(1+q)^{n-1}}$ for all $n \ge 3$.  
\end{lem}

\begin{proof}
First check directly that $D_3$ and $D_4$ satisfy the given
formula. Assume by induction that $D_{n-2}$ and $D_{n-1}$ satisfy the
formula. Applying the recursion \eqref{eq:recursion} gives
\[
D_n = \frac{[n-1]_q}{(1+q)^{n-2}} - \frac{q}{(1+q)^2}
\frac{[n-2]_q}{(1+q)^{n-3}} = \frac{(1+q)[n-1]_q - q[n-2]_q}{(1+q)^{n-1}}
\]
and the result follows from the identity $(1+q)[n-1]_q - q[n-2]_q = [n]_q$.
\end{proof}

Lemma \ref{lem:claim} proves the claim, as $q$ is not a root of
unity. The claim implies that $\C A_1 + \C A_2 + \cdots + \C A_{n-1} =
\sum_{j=1}^{n-1} \C e_{1j}$. Hence $e_{1j} \in \g$ for all $j = 1,
\dots, n-1$. Now we finish the proof quite easily, by observing that
\[
  [e_{ij}, u_{i+1}] =
  \begin{cases}
    -b e_{i+1,j} & \text{if } i+1 \ne j \\
    -b e_{jj} + be_{j-1,j-1} + e_{j-1,j} + ae_{j-1,j+1} & \text{if }
        i+1 = j \ne n-1 \\
    -b e_{n-1,n-1} + be_{n-2,n-2} + e_{n-2,n-1} & \text{if } i+1 = j = n-1.
  \end{cases}
\]
This implies that $[e_{ij}, u_{i+1}] = -b e_{1+1,j}$ modulo a linear
combination of elements of the form $e_{kl}$ where $k < i+1$. Assuming
by induction that the $e_{kl} \in \g$ for all $k < i+1$, it follows
from the fact that $b \ne 0$ that $e_{i+1,j} \in \g$ for all $j$. This
completes the proof of (a) in case $n \ge 4$.

The case $n=3$ must be handled separately, and is simpler.  Notice
that in this case
\[
A'_2 = \tfrac{1}{2a}\big( [u_1,[u_1,u_2]] + [u_1,u_2]\big) = (ab-1)
e_{12} \in \g.
\]
Since $ab-1 = 0$ if and only if $[3]_q = 1+q+q^2 = 0$, we see that
$ab-1 \ne 0$ as $q$ is not a root of unity. Hence $e_{12} \in \g$, and
$e_{11} = u_1 -a e_{12} \in \g$.  Thus $[e_{11},u_2] = -b e_{21} \in
\g$, which implies that $e_{21} \in \g$ and $e_{22} = u_2 - b e_{21}
\in \g$, completing the proof of (a) in the case $n=3$.

(b)
Now let $\g_1=\langle u_1,\ldots u_{n-1}\rangle$ be the Lie algebra in
part (a), and set $\g_2=\langle v_1,\ldots
v_{n-1}\rangle$. Notice that
\[
u_k =\tfrac{1}{n-1}\left( I_{n-1} -v_k\right)
\]
for all $k$. Thus $[u_i,u_j]=c[v_i,v_j]$ with $c=1/(n-1)^2$. Hence
$\g_1'=\g_2'$, that is, the derived algebras of $\g_1$, $\g_2$
coincide. By part (a), $\g_1=\gl_{n-1}$, so
$\g_1'=\gl_{n-1}'=\mathfrak{sl}_{n-1}$. Thus
$\g_2'=\mathfrak{sl}_{n-1} \subset \g_2.$ As it is obvious that
$\g_2\subset \mathfrak{sl}_{n-1}$, it follows that $\g_2
=\mathfrak{sl}_{n-1}$. This completes the proof of Proposition
\ref{prop:Lie-alg}.
\end{proof}

We now have all the tools needed to prove Theorem \ref{thm:density}.

\begin{proof}[Proof of Theorem \ref{thm:density}]
It is well known (see e.g., \cites{Howe,Hall}) that closed subgroups
of $\GL(\bF)$ contain the one-parameter subgroups generated by all
elements of their Lie algebra; indeed, that can be taken as an
equivalent definition of the Lie algebra for such groups. By
Proposition \ref{prop:Lie-alg}, it easily follows that $\ov{G}$
contains $\SL_{n-1} \cong \SL(\bF)$, so we are done. (If $n = 2$ there
is nothing to prove, as $\dim \bF = 1$ in that case.)
\end{proof}

\begin{rmk}
(i) It is easy to see that $\ov{G} = \GL(\bF)$ if and only if
  $q_1^{n-2}q_2$ is not a root of unity. The sufficiency of this
  condition follows from Proposition \ref{prop:Lie-alg}(a). For its
  necessity, observe that if $\xi = q_1^{n-2}q_2$ is a root of unity
  (of order $d$, say) then elements of $G$ satisfy the polynomial
  equation
  \[
    \prod_{p=0}^{d-1} (\det(X_{ij}) - \xi^p) = 0
  \]
  and not all elements of $\GL(\bF)$ satisfy it.

(ii) When $q_1^{n-2} q_2$ is a primitive $d$th root of unity, for
  $d>1$, we have a strict containment $\ov{G} \supsetneqq \SL(\bF)$,
  because the generators of $G$ do not have determinant one.
\end{rmk}


\end{document}